\documentclass[a4paper,10pt]{amsart}
\usepackage[english]{babel}
\usepackage[utf8]{inputenc}
\usepackage[T1]{fontenc}
\usepackage{amssymb}
\usepackage{mathrsfs}
\usepackage[all]{xy}	
\usepackage{hyperref}
\usepackage[usenames,dvipsnames]{xcolor}
\hypersetup{colorlinks,%
citecolor=Black,%
filecolor=Black,%
linkcolor=Black,%
urlcolor=Black}
\usepackage{enumitem}	% Per personalizzare gli elenchi
\usepackage{array}		% Per formattare meglio le celle nelle tabelle.
\allowdisplaybreaks
\usepackage{graphicx}
\usepackage{pdflscape}
\newcommand{\scr}[1]{\mathscr{#1}}
\newcommand{\frk}[1]{\mathfrak{#1}}
\newcommand{\bb}[1]{\mathbb{#1}}

\newcommand{\N}{\mathbb{N}}	% Numeri naturali
\newcommand{\Z}{\mathbb{Z}}	% Numeri interi
\newcommand{\R}{\mathbb{R}}	% Numeri reali
\newcommand{\Id}{\mathrm{Id}}	% Mappa identit√†
\newcommand{\Span}{\mathrm{span}}	% Span
\newcommand{\Ker}{\mathrm{Ker}}	% Ker
\newcommand{\dd}{\,\mathrm{d}}	% 'd' di derivata
\newcommand{\de}{\partial}		% Derivata parziale
\newcommand{\THEN}{\Rightarrow}	% =>
\newcommand{\IFF}{\Leftrightarrow}	% <=>

\newcommand{\ci}{\framebox[\width]{$\subset$} }
\newcommand{\ic}{\framebox[\width]{$\supset$} }

\newcommand{\Lin}{\mathtt{Lin}}
\newcommand{\End}{\mathtt{End}}

\newcommand{\jet}{\mathfrak{j}}
\newcommand{\Jet}{\mathtt{J}}
\newcommand{\HorDer}{\mathtt{HD}}

\newcommand{\Der}{\mathtt{Der}}
\newcommand{\Aut}{\mathtt{Aut}}
\newcommand{\Poly}{\scr P}
\newcommand{\UniEnvAlg}{\scr U}
\newcommand{\Tensor}{\scr T}

\newcommand{\Ad}{\operatorname{Ad}}
\newcommand{\w}{\mathbf{w}}

\newcommand{\rcontr}{\raisebox{\depth}{\scalebox{1}[-1]{$\lrcorner$}}} 

\newcommand{\dual}[1]{#1^*} %{\overset{*}{#1}}
\newcommand{\ts}{\otimes}

\theoremstyle{plain}
\newtheorem{proposition}{Proposition}[section]
\newtheorem{theorem}[proposition]{Theorem}
\newtheorem{lemma}[proposition]{Lemma}
\newtheorem{corollary}[proposition]{Corollary}
\newtheorem{thm}{Theorem}[section]

\theoremstyle{definition}

\newtheorem{remark}[proposition]{Remark}

\theoremstyle{remark}

\title{Jet spaces on Carnot groups}
\author[Nicolussi~Golo]{Sebastiano Nicolussi Golo}
\address[Nicolussi~Golo]{Department of Mathematics and Statistics, 40014 University of Jyväskylä, Finland}	
\email{sebastiano.s.nicolussi-golo@jyu.fi}

\author[Warhurst]{Benjamin Warhurst}
\address[Warhurst]{Institute of Mathematics, University of Warsaw, ul. Banacha 2, 02-097 Warsaw, Poland}
\email{b.warhurst@mimuw.edu.pl}
\thanks{
S.~N.~G.~has been supported by the Academy of Finland (%
grant 328846, ``Singular integrals, harmonic functions, and boundary regularity in Heisenberg groups'',
grant 322898 ``Sub-Riemannian Geometry via  Metric-geometry and Lie-group Theory'',
grant 314172 ``Quantitative rectifiability in Euclidean and non-Euclidean spaces''),
and by the University of Padova STARS Project ``Sub-Riemannian Geometry and Geometric Measure Theory Issues: Old and New''.\\
B.~W.~was supported by the grant of the National Science Center, Poland (NCN), UMO-2017/25/B/ST1/01955.\\
S.~N.~G.~and B.~W.~are grateful for the support of this research provided by the grant of the National Science Center, Poland (NCN), UMO-2017/25/B/ST1/01955.\\
Part of this research was done while S.~N.~G.~was visiting B.~W.~at IMPAN, Warsaw.
The excellent work atmosphere is acknowledged.
}



\date{\today} %. \IfFileExists{./.gittex}{\input{./.gittex}}{}}

\begin{document}
\begin{abstract}
Jet spaces on $\R^n$ have been shown to have a canonical structure of stratified Lie groups (also known as Carnot groups).
We construct jet spaces over stratified Lie groups adapted to horizontal differentiation and show that these jet spaces are themselves stratified Lie groups. Furthermore, we show that these jet spaces support a prolongation theory for contact maps, and in particular, a B\"acklund type theorem holds. A byproduct of these results is an embedding theorem that shows that every stratified Lie group of step $s+1$ can be embedded in a jet space over a stratified Lie group of step $s$.
\end{abstract}
\maketitle

\setcounter{tocdepth}{1}
\phantomsection
\addcontentsline{toc}{section}{Contents}
\tableofcontents

%%%%%%%%%%%%%%%%%%%%%%%%%%%%%%%%%%%%%%%%%%%%%%%%%%%%%%%%%%%%%%%
%%%%%%%%%%%%%%%%%%%%%%%%%%%%%%%%%%%%%%%%%%%%%%%%%%%%%%%%%%%%%%%
\section{Introduction}
%%%%%%%%%%%%%%%%%%%%%%%%%%%%%%%%%%%%%%%%%%%%%%%%%%%%%%%%%%%%%%%
\subsection{Overview}

Classical jet spaces are vector bundles $\Jet^m(\R^n;\R^\ell) \to \R^n$ endowed with a horizontal bundle $\scr H^m\subset T\Jet^m(\R^n;\R^\ell)$,
such that jets of functions $\R^n\to\R^\ell$ are exactly those sections that are tangent to $\scr H^m$.
See \cite{MR989588} for further reference.

It is well known that these jet spaces $\Jet^m(\R^n;\R^\ell)$ carry a structure of stratified Lie group for which $\scr H^m$ is left-invariant,
see \cite{MR2115247}.
\emph{Stratified Lie groups} are also called \emph{Carnot groups} and we use the two terms as synonyms.

We observed that stratified Lie groups of step 2 can be embedded into a jet space $\Jet^1(\R^n;\R^\ell)$ for a suitable choice of $n$ and $\ell$.
By ``embedding'' we mean that there is a strata-preserving injective homomorphism of Lie groups.
At the same time, we realized that stratified Lie groups of step larger than 2 cannot be embedded in such jet spaces, see Remark~\ref{rem617bb0a5}.
A heuristic analysis led us to believe that we needed to substitute the $\R^n$ in $\Jet^m(\R^n;\R^\ell)$ with a stratified Lie group.

In this paper, we consider jets adapted to horizontal differentiation of smooth functions $f$ from a stratified group $G$ into a vector space $W$.
The stratification of $G$ induces an ``intrinsic'' notion of degree both for polynomials and for left-invariant differential operators.
Following~\cite{MR657581}, 
it is natural to reorganize the derivatives of $f$ using this intrinsic degree.
Equivalently, one can reorganize the homogeneous polynomials of the Taylor expansion of $f$.
See also~\cite{Neusser2010Weighted-jet-bu}.

Such a reorganization allows us to construct vector bundles $\Jet^m(G;W)\to G$
endowed with a horizontal bundle $\scr H^m\subset T\Jet^m(G;W)$ that characterizes sections defined as ``horizontal'' jets of functions $G\to W$.
We call such spaces \emph{jet spaces over $G$} and we prove two fundamental facts of jet spaces, that is, a prolongation and a de-prolongation theorem.

Finally, we show that every stratified Lie group of step $s+1$ can be embedded into some $\Jet^s(G';W)$ for a stratified Lie group $G'$ of step $s$.

%%%%%%%%%%%%%%%%%%%%%%%%%%%%%%%%%%%%%%%%%%%%%%%%%%%%%%%%%%%%%%%
\subsection{Detailed description of results}
For a vector space $W$ and a simply connected Lie group $G$ with stratified Lie algebra $\frk g = \bigoplus_{i=1}^sV_i$, bracket generated by $V_1$, our starting point is to consider the space $\HorDer^k(\frk g;W)$ of $k$-multi-linear maps $A$ on $V_1$ given by $k$-th order horizontal differentiation of smooth functions $f:G\to W$ at the neutral element $e\in G$.
More precisely, %$A:V_1^k\to W$ is a multilinear of
 if $v_1, \dots ,v_k \in V_1$ and $\tilde v_1, \dots, \tilde v_k$ denote the corresponding left invariant vector fields on $G$, then 
\[
A (v_1, \dots ,v_k) = \tilde v_k \dots  \tilde v_1 f(e) .
\]
We then define $\HorDer^{\leq m}(\frk g;W)=\bigoplus_{k=0}^m \HorDer^k(\frk g;W)$,
where we set $\HorDer^0(\frk g;W)=W$.

If $G$ is abelian, i.e., $G=\R^n$ for some $n$, then $\HorDer^k(\R^n;W)$ is the space of symmetric $k$-multilinear maps from $\R^n$ to $W$.
If $G$ is not abelian,  we don't have such a clear characterization of $\HorDer^k(\frk g;W)$, but we do have an algorithm to construct a basis of it, see~Remark~\ref{rem6093eb23}.

The jet space $\Jet^m(G;W)$ will be the simply connected Lie group with Lie algebra 
\[
\jet^m(\frk g;W) = \frk g\ltimes \HorDer^{\leq m}(\frk g;W)
\]
where the semi-direct product is given by the representation of $\frk g$ over $\HorDer^{\leq m}(\frk g;W)$ that is induced by \emph{right-contractions}.
If $v\in V_1$ and $A\in\HorDer^k(\frk g;W)$, we define the right contraction of $A$ by $v$ as the $(k-1)$-multi-linear map $v\rcontr A\in \HorDer^{k-1}(\frk g;W)$,
\[
v\rcontr A(v_1,\dots,v_{k-1}) = A(v_1,\dots,v_{k-1},v) .
\]
In Proposition~\ref{prop05301400}, we show that this operation $v\mapsto v\rcontr$ extends to a Lie algebra anti-morphism $\frk g\to\End(\HorDer^{\le m}(\frk g;W))$.

The Lie algebra $\jet^m(\frk g;W)$ turns out to be again stratified of step $\max\{s,m+1\}$, with layers of the form
\[
\begin{array}{rcl}
\jet^m(\frk g;W)_1 & = &V_1 \oplus \HorDer^m(\frk g;W) \\
\jet^m(\frk g;W)_2 & = &V_2 \oplus \HorDer^{m-1}(\frk g;W) \\
\jet^m(\frk g;W)_3 & = &V_3 \oplus \HorDer^{m-3}(\frk g;W) \\
\vdots & & \vdots
\end{array}
\]
The contact structure $\scr H^m$ of the jet space $\Jet^m(G;W)$ is the left invariant distribution defined by the first layer $\jet^m(\frk g;W)_1$.
A smooth function $f:G\to W$ defines a section $\Jet^mf:G\to \Jet^m(G;W)$ in a natural way, and we call $\Jet^mf$ the \emph{(horizontal) jet of $f$}.
We show in Proposition~\ref{prop5eafeb43} that a section $\gamma:G\to \Jet^m(G;W)$ is the jet of a function if and only if $d\gamma$ maps $\scr H_G$ to $\scr H^m$, where $\scr H_G$ is the left-invariant distribution on $G$ defined by $V_1$. 

Once the definition and the structure of $\Jet^m(G;W)$ are set, we prove three main results.
First of all, we demonstrate that contact maps $\Jet^m(G;W)\to \Jet^m(G;W)$ are subject to a prolongation theory similar to the classical case and that a B\"acklund type theorem holds.

\begin{thm}[Prolongation Theorem]\label{thm617406c6}
	Suppose $m\ge0$, $\Omega\subset\Jet^m(G;W)$ is open, and that $F:\Omega\to\Jet^m(G;W)$ is a contact map. If $\pi_m:\Jet^{m+1}(G;W)\to \Jet^{m}(G;W)$ denotes the projection along $\HorDer^{m+1}(\frk g;W)$, then 
	there is an open set $\hat \Omega\subset\Jet^{m+1}(G;W)$
	and a unique contact map $\hat F:\hat \Omega\to\Jet^{m+1}(G;W)$ such that 
	\[
	\pi_m\circ \hat F = F\circ\pi_m. 
	\]
\end{thm}

A more precise statement is Theorem~\ref{thm5eafe9c5}, where we give a precise definition of the open set $\hat\Omega$.
The size of the set $\hat\Omega$ remains unclear to us in general.
For sure, if $F$ is close enough to the identity map, then $\hat\Omega$ is not empty.
Moreover, if $F$ is itself the prolongation of a contact map, then $\hat\Omega = \pi_m^{-1}(\Omega)$, see Remark~\ref{rem61765cb1}.
As a consequence of the following Theorem~\ref{thm5ebc31f5}, this is always the case when $m\ge2$.
We don't know if it can be proven that $\pi_m(\hat\Omega) = \Omega$ also when $m=1$.

\begin{thm}[de-Prolongation Theorem]\label{thm5ebc31f5}
	If $m\ge2$, then every contact diffeomorphism $\Jet^m(G;W)\to\Jet^m(G;W)$ is the prolongation of a contact diffeomorphism $\Jet^1(G;W)\to\Jet^1(G;W)$.
		
	Moreover, suppose that one of the following conditions is satisfied:
	\begin{enumerate}[label=(\Alph*)]
	\item%\label{item5ebc765a}
	$\dim(W)>1$, or
	\item%\label{item5ebc765b}
	for every $v\in V_1\setminus\{0\}$ there is $v'\in V_1$ with $[v,v']\neq0$.
	\end{enumerate}
	Then every contact diffeomorphism $\Jet^1(G;W)\to\Jet^1(G;W)$ is the prolongation of a contact diffeomorphism $\Jet^0(G;W)\to\Jet^0(G;W)$.
\end{thm}

We can also state Theorem~\ref{thm5ebc31f5} for contact diffeomorphisms that are defined on domains in $\Jet^m(G;W)$, in which case the de-prolongation will only be local a priori.
A precise result is stated in Theorems~\ref{thm5ebc3248} and~\ref{thm606f2a10}.
We show in Remark~\ref{rem616fabbb} that the second part of Theorem~\ref{thm5ebc31f5} is sharp.

Finally, we prove that stratified groups embed into jet spaces.
It has already been proved by Montgomery in~\cite[\S6.5.1]{MR1867362} that every manifold of dimension $n$ endowed with a distribution of rank $k$ is \emph{locally smoothly embeddable} into $\Jet^1(\R^k;\R^{n-k})$.
Our result differ from Montgomery's as we seek a \emph{global embedding as stratified Lie groups}, that is, we want to reconstruct a stratified Lie group $G$ as a stratified subgroup of some jet space.
This is simply not possible with standard jet spaces, as we explain in Remark~\ref{rem617bb0a5}.

Notice that, of course, $G$ embeds into $\Jet^m(G;W)$, by construction.
However, we prove that we can reconstruct $G$ inside the jet space of a stratified group with lower step.

\begin{thm}[Embedding into jet spaces]\label{thm6171c00e}
	Let $G$ be a stratified group of step $s+1$.
	Then there exists a stratified group $G'$ of step $s$ and a vector space $W$ such that $G$ embeds as a stratified group in $\Jet^s(G';W)$.
\end{thm}

%%%%%%%%%%%%%%%%%%%%%%%%%%%%%%%%%%%%%%%%%%%%%%%%%%%%%%%%%%%%%%%
\subsection{Structure of the paper}
In the preliminary section~\ref{sec617bb273}, we introduce elementary facts, notation and conventions we will use throughout the paper.
The definition of jet spaces and their Lie algebras over stratified Lie groups are contained in Section~\ref{sec617799fa}.
Section~\ref{sec616fbbc5} is devoted to a second definition of jet spaces using homogeneous polynomials.
We prove the Prolongation Theorem~\ref{thm617406c6} in Section~\ref{sec610264b1} and de-Prolongation Theorem~\ref{thm5ebc31f5} in Section~\ref{sec616fa813}.
Section~\ref{sec61719991} is devoted to the proof of the Embedding Theorem~\ref{thm6171c00e}.
Finally, we present one example in Section~\ref{sec617bad0c}.

%%%%%%%%%%%%%%%%%%%%%%%%%%%%%%%%%%%%%%%%%%%%%%%%%%%%%%%%%%%%%%%
\section{Preliminaries}
\label{sec617bb273}
In this section we introduce stratified Lie groups and we
 fix basic notation, conventions and a few known results.

%%%%%%%%%%%%%%%%%%%%%%%%%%%%%%%%%%%%%%%%%%%%%%%%%%%%%%%%%%%%%%%
\subsection{Notation choices}
If $V$ and $W$ are vector spaces, we denote by $\Lin^k(V;W)$ the vector space of all $k$-multilinear maps from $V$ to $W$.
In particular, $\End(V)=\Lin(V;V)$ is the Lie algebra of all linear maps $V\to V$ with Lie brackets
\begin{equation}\label{eq6176cbe5}
[A,B] = AB-BA .
\end{equation}
If $E\to G$ is a vector bundle over a manifold $G$, we denote by $\Gamma(E)$ the space of smooth sections of $E$.
If $X\in\Gamma(E)$ and $p\in M$, we write $X(p)$ or $X_p$ for the evaluation of $X$ at $p$.
If $G,M$ are manifolds, we denote by $C^\infty(G;M)$ the space of smooth maps $G\to M$.
If $G$ is a Lie group, we denote its identity element by $e_G$ or $e$, 
and we identify its Lie algebra with the tangent space at the identity element;
if $p\in G$, then $L_p:G\to G$ is the left translation $L_p(x)=px$.
If $\frk g$ is a Lie algebra, we denote by $\Der(\frk g)$ the space of derivations of $\frk g$ and by $\Aut(\frk g)$ the group of Lie algebra automorphisms of $\frk g$.
If $G$ is a Lie group, we denote by $\Aut(G)$ the group of Lie group automorphisms of $G$.

If $V$ is a vector space, let $\Tensor(V) = \bigoplus_{j=1}^\infty V^{\otimes j}$ be the tensor algebra of $V$.
Define the tensor spaces $\Tensor^m(V) = V^{\otimes m}$ and  $\Tensor^{\le m}(V) = \bigoplus_{k=0}^m \Tensor^k(V)$.
It is clear that $\Tensor^m(V)^* = \Tensor^m(V^*)$.

%%%%%%%%%%%%%%%%%%%%%%%%%%%%%%%%%%%%%%%%%%%%%%%%%%%%%%%%%%%%%%%
\subsection{Vector fields}\label{sec61710124}
Let $G$ be a smooth manifold and $W$ a finite dimensional real vector space.
We consider the differential of a smooth map $f:G\to W$ as a map $\dd f:TG\to W$.
More precisely, if $v\in T_pG$, then $\dd f(v) = \frac{\dd}{\dd t}|_{t=0} f(\gamma(t)) \in W$ for any $\gamma:\R\to G$ smooth with $\gamma(0)=p$ and $\gamma'(0)=v$.

Vector fields $\tilde v\in\Gamma(TG)$ 
are differential operators $C^\infty(G;W)\to C^\infty(G;W)$:
if $f:G\to W$ is smooth, then $\tilde vf:G\to W$ is the map $\tilde vf(p) = \dd f(\tilde v(p))$.
The Lie bracket of vector fields $\tilde v,\tilde w\in\Gamma(TG)$ is defined as 
the commutator in the algebra of differential operators, i.e., 
\begin{equation}\label{eq6176ca3e}
[\tilde v,\tilde w] = \tilde v\tilde w - \tilde w\tilde v .
\end{equation}

%%%%%%%%%%%%%%%%%%%%%%%%%%%%%%%%%%%%%%%%%%%%%%%%%%%%%%%%%%%%%%%
\subsection{Vector fields on a Lie group}\label{sec616e79a0}
Let $G$ be a Lie group with Lie algebra $\frk g$.
If $X\in C^\infty(G;\frk g)$, i.e., $X$ is a smooth function $G\to \frk g$, we denote by $\tilde X$ the corresponding vector field on $G$ defined by $\tilde X(p) = \dd L_p[X(p)]$.
The vector field $\tilde X$ is characterized by means of the Mauer--Cartan form $\omega_G$ as the only vector field $\tilde X\in\Gamma(TG)$ such that $\omega_G(\tilde X(p))=X(p)$ for all $p\in G$.
We define Lie brackets on $C^\infty(G;\frk g)$ as follows:
If $X,Y:G\to\frk g$, then we define $[X,Y](p) = \dd L_{p}^{-1}[\tilde X,\tilde Y](p) \in\frk g$.
Notice that left-invariant vector fields $\tilde X:G\to\Gamma(TG)$ correspond to constant functions $G\to\frk g$, and that if $X,Y:G\to\frk g$ are constant then $[X,Y](p) = [X(p),Y(p)]$ for all~$p$.

If $f,g\in C^\infty(G)$ and $X,Y:G\to\frk g$ are constant, then
\begin{equation}\label{eq616e7ae8}
[fX,gY] = f\cdot(\tilde Xg)\cdot Y - g\cdot(\tilde Yf)\cdot X + fg [X,Y] .
\end{equation}

Sometimes we will use right-invariant vector fields.
For $x\in\frk g$, we denote by $x^\dagger$ the right-invariant vector field on $G$ with $x^\dagger(e_G)=x$.

%%%%%%%%%%%%%%%%%%%%%%%%%%%%%%%%%%%%%%%%%%%%%%%%%%%%%%%%%%%%%%%
\subsection{Derivatives along vectors and linear vector fields}
Let $V$ be a vector space.
As vector spaces are instances of abelian Lie groups, we are consistent with the notation introduced in Section~\ref{sec616e79a0} by setting
\[
\tilde vf(p) = \frac{\dd}{\dd t}|_{t=0} f(p+tv)
\]
for any vector $v\in V$.

We will need the following elementary observation.
Notice that a linear map $X\in\End(V)$ defines a vector field $\tilde X(p) = Xp$.
So, if $X,Y\in\End(V)$, then we have two ways of taking Lie brackets:
as linear maps in~\eqref{eq6176cbe5}, $[X,Y]_{\End(V)} = XY-YX\in\End(V)$; 
as vector fields in~\eqref{eq6176ca3e}, $[\tilde X,\tilde Y]_{\Gamma(TV)}=\tilde X\tilde Y-\tilde Y\tilde X \in\Gamma(TV)$.

\begin{lemma}\label{lem05291925}
	If $X,Y\in\End(V)$, then $[\tilde X,\tilde Y]_{\Gamma(TV)} = - [X,Y]_{\End(V)}^{\widetilde{~}}$.
\end{lemma}
\begin{proof}
	Let $f:V\to\R$ be a smooth function and $p\in V$.
	Notice that
	\begin{align*}
	\tilde X\tilde Yf (p)
	&= \left.\frac{\dd}{\dd s}\right|_{s=0} (\tilde Yf)(p+sXp) 
	= \left.\frac{\dd}{\dd t}\right|_{t=0}\left.\frac{\dd}{\dd s}\right|_{s=0} 
		f((p+sXp)+tY(p+sXp)) \\
	&= \left.\frac{\dd}{\dd t}\right|_{t=0}\left.\frac{\dd}{\dd s}\right|_{s=0} 
		f(p+sXp + tYp+st YXp) \\
	&= \left.\frac{\dd}{\dd t}\right|_{t=0}\left.\frac{\dd}{\dd s}\right|_{s=0}
		(f(p+sXp+tYp) + f(p+tYp+st YXp)) \\
	&= \left.\frac{\dd}{\dd t}\right|_{t=0}\left.\frac{\dd}{\dd s}\right|_{s=0}
		(f(p+sXp+tYp) + f(p+tYp) + f(p+st YXp)) \\
	&= \left.\frac{\dd}{\dd t}\right|_{t=0}\left.\frac{\dd}{\dd s}\right|_{s=0}
		(f(p+sXp+tYp) + f(p+tYp) + s f(p+t YXp)) ,
	\end{align*}
	where we used the formula $\left.\frac{\dd}{\dd s}\right|_{r=0}g(r,r) = \left.\frac{\dd}{\dd r}\right|_{r=0} (g(r,0) + g(0,r))$.
	Similarly, 
	\[
	\tilde Y\tilde Xf (p) = \left.\frac{\dd}{\dd t}\right|_{t=0}\left.\frac{\dd}{\dd s}\right|_{s=0}
		(f(p+sXp+tYp) + f(p+tYp) + s f(p+t XYp)) .
	\]
	Hence,
	\begin{align*}
	[\tilde X,\tilde Y]f(p)
	&= \tilde X\tilde Yf (p) -\tilde Y\tilde Xf(p) \\
	&= \left.\frac{\dd}{\dd t}\right|_{t=0}
		(f(p+t YXp)-f(p+t XYp))) \\
	&= \left.\frac{\dd}{\dd t}\right|_{t=0} f(p+ t(YXp-XYp)) \\
	&= (-[X,Y]p)^{\widetilde{~}}f(p) .
	\end{align*}
\end{proof}

%%%%%%%%%%%%%%%%%%%%%%%%%%%%%%%%%%%%%%%%%%%%%%%%%%%%%%%%%%%%%%%
\subsection{Vector valued differential forms}
Let $G$ be a smooth manifold and $W$ a vector space.
We define \emph{$W$-valued 1-forms} as
\[
\Omega^1(G;W) := \Gamma(\Lin(TG;W)),
\]
where $\Lin(TG;W)$ is the vector bundle over $G$ whose fiber at $p\in G$ is the space $\Lin(T_pG;W)$ of linear maps from $T_pG$ to $W$.

Suppose that $G$ is a Lie group with Lie algebra $\frk g$.
Via the Mauer--Cartan form, we identify $TG$ with $G\times\frk g$ and thus $\Lin(TG;W)$ with $G\times\Lin(\frk g;W)$.
In particular, 
\[
\Omega^1(G;W) \simeq C^\infty(G;\Lin(\frk g;W)) .
\]
In other words, if $\omega\in C^\infty(G;\Lin(\frk g;W))$, then we define
$\tilde \omega\in\Omega^1(G;W)$ as $\tilde\omega(p)(v) = \omega(p)(\dd L_p^{-1}v)$ for $v\in T_pG$.

A \emph{left-invariant $W$-valued 1-form} is an element $\tilde\omega$ of $\Omega^1(G;W)$ such that, for all $p,g\in G$, $\tilde\omega(gp)=\tilde\omega(p)\circ\dd L_g^{-1}$.
In particular, left-invariant $W$-valued 1-forms correspond to constants in $C^\infty(G;\Lin(\frk g;W))$.

Any $W$-valued 1-form $\tilde\omega\in\Omega^1(G;W)$ defines a subset $\Ker(\tilde\omega) \subseteq TG$ given by all vectors $v\in TG$ such that $\tilde\omega(v)=0$. In particular, $\Ker(\tilde\omega)$ is a subbundle of $TG$ if $\tilde\omega$ is left-invariant.

%%%%%%%%%%%%%%%%%%%%%%%%%%%%%%%%%%%%%%%%%%%%%%%%%%%%%%%%%%%%%%%
\subsection{Anti-semi-direct product}\label{app05301812}
Let $\frk g$ and $\frk h$ be Lie algebras, and $\psi:\frk g\to\Der(\frk h)$ an anti-morphism of Lie algebras, that is, a linear map such that 
$\psi([x,y]) = -[\psi(x),\psi(y)]$. For a given $\psi$, we obtain a Lie algebra $\jet=\frk g\ltimes_\psi\frk h$ consisting of the vector space $ \frk g \times \frk h$ equipped with the bracket 
\[
[(x,X),(y,Y)] = ([x,y],-[X,Y]+\psi(y)X-\psi(x)Y ).
\]

Let $G$ and $H$ be the corresponding connected simply connected Lie groups of $\frk g$ and $\frk h$. Then there is a anti-morphism of Lie groups $\phi:G\to\Aut(H)$ (that is, $\phi(ab)=\phi(b)\phi(a)$) such that
$\phi(\exp(x))_* = e^{\psi(x)} \in \Aut(\frk h)$, for all $x\in\frk g$. We obtain Lie Group $\Jet=G\ltimes_\phi H$ on $G \times H$ with Lie algebra $\jet=\frk g\ltimes_\psi\frk h$ by setting
$(a,A)(b,B) = (ab , B\phi(b)A )$.

Assume that $\frk h$ is abelian, that is, a finite-dimensional vector space.
Then
\begin{align*}
	[(x,X),(y,Y)] &= ([x,y], \psi(y)X-\psi(x)Y ) , \\
	(a,A)(b,B) &= (ab , B+\phi(b)A ) , \\
	(\exp(x),A)(\exp(y),B) &= (\exp(x)\exp(y) , B+ e^{\psi(y)} A ) .
\end{align*}
If $(a,A)\in\Jet$ and $(x,X)\in\jet$, then the differential of the left translation $L_{(a,A)}$ at $(e_G,0)$ applied to $(x,X)$ is
\[
dL_{(a,A)}|_{(e_G,0)} \begin{pmatrix}x\\X\end{pmatrix}
= \begin{pmatrix} dL_a & 0 \\ \psi(\cdot) A & \Id \end{pmatrix} \begin{pmatrix}x\\X\end{pmatrix}
= \begin{pmatrix} \tilde x(a) \\ \psi(x)A + X \end{pmatrix} .
\]
Hence
if $(x,X)\in\jet$, the corresponding left-invariant vector field 
$(x,X)^{\widetilde{~}}$ computed at $(a,A)\in\Jet$ is
\[
(x,X)^{\widetilde{~}} (a,A) = ( \tilde x(a) , \psi(x)A + X ) .
\]
In order to describe the exponential map $\exp_\Jet:\jet\to\Jet$, 
we need to find a curve $t\mapsto (a(t),A(t))\in\Jet$ such that 
$a(0)=e_G$, $A(0)=0$, $a'(t) =  \tilde x(a(t))$ and $A'(t) = \psi(x)A(t) + X$.
By standard considerations, we get
\begin{equation}\label{eq6171635c}
\exp_\Jet(x,X) = \left( \exp(x) , \int_0^1 e^{(1-s)\psi(x)}X \dd s \right) .
\end{equation}

\subsection{Stratified Lie groups}
A \emph{stratification} for a Lie algebra $\frk g$ is a direct sum decomposition $\mathfrak{g} = V_1 \oplus \cdots \oplus V_s$ where $\lbrack V_1, V_j \rbrack = V_{1+j}$ for $1\leq j \leq s$, and $\lbrack V_i, V_j \rbrack = 0$ for $i + j > s$.  
The projection $\frk g\to V_i$ will be denoted by $\Pi_i$.
A \emph{stratified Lie group} is a connected, simply connected nilpotent Lie group whose Lie algebra is stratified.
Stratified Lie groups are also known as \emph{Carnot groups} and we use the two terms as synonyms.

The subbundle of $TG$ given pointwise by $\{ \tilde v(p) \in T_pG : v \in V_1,\ p\in G\}$ is called the \emph{horizontal bundle} and denoted $\scr H_G$. 
Furthermore, for every $\lambda>0$, the \emph{dilation} $\delta_\lambda: \mathfrak{g} \to \mathfrak{g}$, is the Lie algebra isomorphism  defined by $\Pi_i \circ \delta_\lambda = \lambda^i  \Pi_i$. 
By conjugating the Lie algebra dilation with the exponential map, we get a corresponding dilation $G \to G$, also denoted $\delta_\lambda$, which by  definition, is a group isomorphism.  

\begin{lemma}\label{lem6177fbb7}
	Let $F:G\to G$ be a smooth map with $dF(\scr H_G)\subset\scr H_G$ and let $p_0\in G$ be such that $dF|_{\scr H_G(p_0)}$ is injective.
	Then $dF(p_0)$ is a linear isomorphism.
\end{lemma}
\newcommand{\GL}{\mathtt{GL}}
\begin{proof}
	We say that two vector fields $\tilde X_1,\tilde X_2$ are $F$-related if 
	\[
	dF(p)\tilde X_1(p) = \tilde X_2(F(p))
	\]
	for all $p$.
	If $\tilde X_1,\tilde Y_1$ are $F$-related to $\tilde X_2,\tilde Y_2$, respectively, then $[\tilde X_1,\tilde Y_1]$ is $F$-related to $[\tilde X_2,\tilde Y_2]$, see~\cite[Proposition~1.55]{MR722297}.
	
	Let $U$ be a neighborhood of $p_0$ where $dF|_{\scr H_G(p)}$ is injective for all $p\in U$.
	Define $d_GF(p):\frk g\to\frk g$ by setting
	\[
	d_GF(p)v = dL_{F(p)}^{-1}(dF(p)(dL_pv)) .
	\]
	Since $dF(\scr H_G)\subset\scr H_G$ and $dF|_{\scr H_G(p)}$ is injective for all $p\in U$, then we have a smooth function
	$d_GF|_{V_1} : U\to \GL(V_1)$, where $\GL(V_1)$ is the group of invertible linear maps $V_1\to V_1$.
	Given $X_2:U\to V_1$, define $X^F_1:U\to V_1$ as 
	\[
	X^F_2(p) = d_GF(p)|_{V_1}^{-1}X_2(F(p)) .
	\]
	It is clear that the corresponding vector fields $\tilde X^F_1$ and $\tilde X_2$ are $F$-related.
	
	Now, since there is a basis of $T_{F(p_0)}G$ whose elements are evaluations at $p_0$ of iterated Lie brackets of some horizontal vector fields, then each element of such a basis is in the image of $dF(p_0)$.
	Therefore, $dF(p_0)$ is surjective, and thus a linear isomorphism.
\end{proof}

%%%%%%%%%%%%%%%%%%%%%%%%%%%%%%%%%%%%%%%%%%%%%%%%%%%%%%%%%%%%%%%
\section{The Jet Space Over a Carnot Group}\label{sec617799fa}

\subsection{Horizontal derivatives}
For $f:G\to W$ smooth, $k\in\N$ and $p\in G$, we define $A^k_{f,p}$ as the $k$-multi-linear map from $V_1$ to $W$ defined by
\[
	A_{f,p}^k(v_1, \dots ,v_k) = \tilde v_k \dots  \tilde v_1 f(p) .
\]
Notice that, if we define $f_p(x)=f(px)$, then
the left invariance of the vector fields $\tilde v_i$ 
implies that $A^k_{f,p} = A^k_{f_p,e}$.

For $k\ge1$, we denote by $\HorDer^k(\frk g;W)$ the vector subspace of $\Lin^k(V_1;W)$ of all such $k$-multi-linear maps.
For $k=0$, we set $\HorDer^0(\frk g;W)=W$ coherently with the definition of $A^0_{f,p}$. We then define 
\[
\HorDer^{\leq m}(\frk g;W)=\bigoplus_{k=0}^m \HorDer^k(\frk g;W).
\]
Elements of $\HorDer^{\le m}(\frk g;W)$ are a priori only sums of derivatives of functions $G\to W$.
Corollary~\ref{cor5eafe914} will show that every element in $\HorDer^{\le m}(\frk g;W)$ is the derivative of a function.

%%%%%%%%%%%%%%%%%%%%%%%%%%%%%%%%%%%%%%%%%%%%%%%%%%%%%%%%%%%%%%%
\subsection{Right Contractions}
 If $k \geq 2$ and  $A\in\Lin^k(V_1;W)$, then for $v\in V_1$, the \emph{right contraction} of $A$ by $v$, denoted $v\rcontr A$, is the element in $\Lin^{k-1}(V_1;W)$ given by
\begin{equation}\label{eq5ebb1101}
v\rcontr A (v_1,\dots,v_{k-1}) =  \, A(v_1,\dots,v_{k-1},v).
\end{equation}
Moreover, if $k = 1$ then $v\rcontr A=A(v)$, and if $k=0$ then $v\rcontr A=0$. 
\begin{lemma}\label{lem5eabebe4}
	If $A\in\HorDer^{k}(\frk g;W)$ and $v\in V_1$, then $v\rcontr A\in\HorDer^{k-1}(\frk g;W)$.
	Furthermore, if $k \geq 1$ and $A=A_{f,p}^{k}$, then it follows that 
	\begin{equation}\label{eq05161949}
	v\rcontr A = \left. \frac{\dd}{\dd t} \right|_{t=0} A_{f,p\exp(tv)}^{k-1}.
	\end{equation}
\end{lemma}
\begin{proof}
	The first assertion of the Lemma is a consequence of \eqref{eq05161949} since  $t\mapsto A_{f,p\exp(tv)}^{k-1}$ is a smooth curve in the vector space $\HorDer^{k-1}(\frk g;W)$.
	
	Turning to the proof of \eqref{eq05161949}, we note first that if $v_1,v_2\in V_1$ then
	\begin{align*}
		\tilde v_2\tilde v_1 f_p(y)
		&= \tilde v_2 (\tilde v_1 f_p)(y) \\
		&= \left.\frac{\dd}{\dd s_2}\right|_{s_2=0} (\tilde v_1 f_p)(y\exp(s_2v_2)) \\
		&= \left.\frac{\dd}{\dd s_2}\right|_{s_2=0} \left.\frac{\dd}{\dd s_1}\right|_{s_1=0} f(py\exp(s_2v_2)\exp(s_1v_1)) .
	\end{align*}
	Iterating this procedure, we obtain
	\begin{equation}\label{eq05162009}
	\tilde v_k\cdots\tilde v_1 f_p(e) 
	= \left.\frac{\dd}{\dd s_1}\right|_{s_1=0} \cdots \left.\frac{\dd}{\dd s_k}\right|_{s_k=0} 
	f(p\exp(s_kv_k)\cdots\exp(s_1v_1)) ,
	\end{equation}
	where the derivations $\left.\frac{\dd}{\dd s_i}\right|_{s_i=0}$ commute with each other.
	We obtain \eqref{eq05161949} by direct calculation as follows.
	If $k=1$ then \eqref{eq05161949} is trivial since $A^0_{f,p} = f(p)$. If $k \geq 2$ then 
	\begin{align*}
	&\left. \frac{\dd}{\dd t} \right|_{t=0} A_{f,p\exp(tv)}^{k-1} (v_1,\dots,v_{k-1}) \\
	&= \left. \frac{\dd}{\dd t}\right|_{t=0} \left.\frac{\dd}{\dd s_1}\right|_{s_1=0} \cdots \left.\frac{\dd}{\dd s_{k-1}}\right|_{s_{k-1}=0} 
		f(p\exp(tv)\exp(s_{k-1}v_{k-1})\cdots\exp(s_1v_1)) \\
	&= \tilde v\tilde v_{k-1}\cdots \tilde v_1 f(p)
	= A_{f,p}^{k}(v_1,\dots,v_{k-1},v) 
	= v\rcontr A(v_1,\dots,v_{k-1}) .
	\end{align*}
\end{proof}

\begin{lemma}\label{lem5ead3862}
	The linear span of the image of the bilinear map 
	$\rcontr:V_1\times\HorDer^{k}(\frk g;W) \to \HorDer^{k-1}(\frk g;W)$
	is $\HorDer^{k-1}(\frk g;W)$, for all $k\ge1$.
\end{lemma}
\begin{proof} 
	Let $\mathcal{H} = \Span\{v\rcontr A:v\in V_1,\ A\in \HorDer^{k}(\frk g;W)\}$. For a given $f\in C^\infty(G;W)$, define $\psi_f:G\to \HorDer^{k-1}(\frk g;W)$ by $\psi_f(p)=A_{f,p}^{k-1}$. 
	If $\gamma$ is a smooth horizontal path in $G$ and $s_0$ is in the domain of $\gamma$, then there is $v\in V_1$ such that $\gamma'(s_0) = \tilde v(\gamma(s_0))$.
	Hence, by \eqref{eq05161949},
	\begin{equation}\label{HLift}
	\left.\frac{\dd}{\dd s}\right|_{s=s_0} \psi_f(\gamma(s))
	= \left.\frac{\dd}{\dd t}\right|_{t=0} \psi_f(\gamma(s_0)\exp(tv))
	= v(s_0)\rcontr A_{f,\gamma(s_0)}^{k}.
	\end{equation}
	Since $V_1$ bracket generates $\frk g$, every pair of distinct points $q_1,q_2\in G$ are connected by a horizontal curve and \eqref{HLift} implies that $A_{f,q_1}^{k-1} - A_{f,q_2}^{k-1} \in \mathcal{H}$.
	
	For any $B\in\HorDer^{k-1}(\frk g;W)$ and $p \ne e_G$, there exists $f$ such that $B=A^{k-1}_{f,p}$. Let $\phi$ be a smooth cut off function with value $1$ on a neighborhood of $p$ and $0$ on a neighborhood of $e_G$. Applying the argument above, we have $\psi_{\phi f}(p)  - \psi_{\phi f}(e_G) \in \mathcal{H}$, which implies  $\psi_{\phi f}(p) =B \in \mathcal{H}$ since $\psi_{\phi f}(e_G)=0$.
\end{proof}

Given $v\in V_1$, we extend the definition of the right contraction by $v$ to a map 
$v\rcontr:\HorDer^{\le m}(\frk g;W) \to \HorDer^{\le  m }(\frk g;W)$,
where for each $A\in\HorDer^{\le m}(\frk g;W)$, we set 
$(v\rcontr A)^k = v\rcontr A^{k+1}$ for $k<m$ and $(v\rcontr A)^m=0$.

In the following proposition,
we extend the notion of right contraction $v\rcontr A$ to vectors $v$ that are not in $V_1$.
The spirit of such an extension is based on two observations.
On the one hand, vectors in a higher layer $V_k$ are derivations of order~$k$.
On the other hand, higher order derivations are compositions of derivations of order 1.
A look at the example in Section~\ref{sec617bad0c}, and Section~\ref{sec62d95272} in particular, can help the reader to clarify the picture.

\begin{proposition}\label{prop05301400}
	The map $V_1\to \End(\HorDer^{\le m}(\frk g;W)))$, $v\mapsto v\rcontr$, extends uniquely 
	to a Lie algebra anti-morphism $\frk g\to \End(\HorDer^{\le m}(\frk g;W))$, 
	i.e., to a linear map that  satisfies
	\[
	[v,w]\rcontr 
	= - [v\rcontr,w\rcontr] 
	\] 
	for every $v,w\in\frk g$. Furthermore, the map $v\mapsto v\rcontr$ satisfies the following properties:
	\begin{enumerate}[label=(\roman*)]
	\item\label{item5eabe6df}%\label{item2}
		If $v\in V_\ell$ and $A\in \HorDer^{\le m}(\frk g;W)$, then 
		\begin{equation*}%\label{eq05301141}
		(v\rcontr A)^k = 0 \qquad\forall k > m-\ell ,
		\end{equation*}
		and if $A=A_{f,p}^{\le m}$ for some $f$ and $p$, then
		\[
		(v\rcontr A)^k = \left. \frac{\dd}{\dd t} \right|_{t=0} A_{f,p\exp(tv)}^k 
		\qquad\forall k \le m-\ell . 
		\] 
		\item\label{item5eabe96a}%\label{item3}
		If $v\in V_a$, $w\in V_b$, $A\in \HorDer^{\le m}(\frk g;W)$, and if $A=A_{f,p}^{\le m}$ for some $f$ and $p$, then
		\begin{equation*}%\label{eq05301138}
			\left. \frac{\dd}{\dd t} \right|_{t=0} A_{f,p\exp(t[v,w])}^k 
			= ( w\rcontr(v\rcontr A) - v\rcontr(w\rcontr A) )^k ,
			\qquad\forall k\le m-a-b .
		\end{equation*}
	\end{enumerate}
\end{proposition}
\begin{proof}
	We will show that there exists an extension of $v\mapsto v\rcontr$ satisfying~\ref{item5eabe6df} and~\ref{item5eabe96a}, and then we will show that this extension is a Lie algebra anti-morphism.
	
	We start with the map $v\mapsto v\rcontr$ defined only for $v\in V_1$.
	Recall that, if $f:G\to W$ and $v,w\in\frk g$, then 
	{\small
	\begin{equation}\label{eq5eabfac0}
	\left. \frac{\de^2}{\de s\de t} \right|_{s,t=0} \!\!( f(\exp(sv)\exp(tw))\! -\! f(\exp(tw)\exp(sv)) )
	= \left. \frac{\dd}{\dd h} \right|_{h=0} \!\!f(\exp(h[v,w])).
	\end{equation}
	}
	Moreover, if $v,w\in V_1$, then
	\begin{equation}\label{eq5eabf9c5}
	\begin{aligned}
	\left. \frac{\de^2}{\de s\de t} \right|_{s,t=0} A_{f,p\exp(sv)\exp(tw)}^k 
	&= \left. \frac{\de}{\de s} \right|_{s=0} (w\rcontr A_{f,p\exp(sv)}^k) \\
	&= w\rcontr\left( \left. \frac{\de}{\de s} \right|_{s=0}A_{f,p\exp(sv)}^k \right) \\
	&= w\rcontr (v\rcontr A_{f,p}^k) ,
	\end{aligned}
	\end{equation}
	where we used Lemma~\ref{lem5eabebe4} and the linearity of $w\rcontr$.
	It follows that~\ref{item5eabe96a} holds for $a=b=1$.
	Clearly,~\ref{item5eabe6df} holds for $\ell=1$, by Lemma~\ref{lem5eabebe4} and the definition of $\rcontr$.
	
	Let $1\le n\le s-1$ and suppose that we have already extended $\rcontr$ to a bilinear map 
	\[
	\rcontr : \bigoplus_{j=1}^n V_j \times\HorDer^{\le m}(\frk g;W) \to \HorDer^{\le m}(\frk g;W)
	\]
	satisfying both~\ref{item5eabe6df} for $\ell\le n$ and~\ref{item5eabe96a} for $a,b\le n$. For $u\in V_{n+1}$ and $A\in \HorDer^{\le m}(\frk g;W)$, define
	\[
	(u\rcontr A)^k := \left. \frac{\dd}{\dd t} \right|_{t=0} A_{f,p\exp(tu)}^k 
	\]
	for $k\le m-n-1$, and $(u\rcontr A)^k := 0$ for $k\ge m-n$,
	where $f$ determines $A$ at some~$p$.\\
	If $u=[v,w]\in V_{n+1}$ with $v\in V_1$ and $w\in V_n$,
	then~\ref{item5eabe96a} implies that $(u\rcontr A)^k = ( w\rcontr(v\rcontr A) - v\rcontr(w\rcontr A) )^k$ for $k\ge m-n$; hence, $u\rcontr A$ does not depend on the choice of $f$ and $p$.
	Since $u\mapsto u\rcontr A$ is linear and $V_{n+1}$ is linearly generated by $[V_1,V_n]$, $u\rcontr A$ does not depend on the choice of $f$ and $p$ for every $u\in V_{n+1}$.
	
	The newly defined bilinear map
	\[
	\rcontr : \bigoplus_{j=1}^{n+1} V_j \times\HorDer^{\le m}(\frk g;W) \to \HorDer^{\le m}(\frk g;W)
	\]
	satisfies~\ref{item5eabe6df} for $\ell\le n+1$ by construction.
	If $v\in V_a$ and $w\in V_b$ with $a,b\le n+1$, then we can carry out the same computations as in~\eqref{eq5eabf9c5} and apply~\eqref{eq5eabfac0}.
	Thus, we obtain that this partial extension $\rcontr$ also satisfies~\ref{item5eabe96a} for $a,b\le n+1$.
	
	By iteration, we have proven that an extension of $\rcontr$ satisfying~\ref{item5eabe6df} and~\ref{item5eabe96a} exists.
	Finally,~\ref{item5eabe96a} implies that $v\mapsto v\rcontr$ is a Lie algebra anti-morphism, which must be unique because $V_1$ Lie generates $\frk g$.
\end{proof}

\begin{remark}
	From part~\ref{item5eabe6df} of Proposition~\ref{prop05301400},
	we see that if $v\in V_\ell$ and $A\in \HorDer^k(\frk g;W)$, then $v\rcontr A\in \HorDer^{k-\ell}(\frk g;W)$ (or $0$ if $k-\ell<0$).
	In particular, $A\in \HorDer^k(\frk g;W)$, which is a $k$-linear map on $V_1$, defines a linear map $V_k\to W$.
	This observation is an expression of the known fact that derivatives along $V_k$ are ``intrinsically'' derivatives of order $k$.
	
	More generally, $A\in\HorDer^{\le m}(\frk g;W)$ defines a linear map $A:\frk g\to\HorDer^{\le m-1}(\frk g;W)$ by setting $A(v) = v\rcontr A$.
	From~\ref{item5eabe6df} in Proposition~\ref{prop05301400}, we see that 
	if $A=A^{\le m}_{f,p}$ and $k+ s\le m$, then
	\begin{equation}\label{eq6177b5e3}
	(v\rcontr A)^k = \left.\frac{\dd}{\dd t}\right|_{t=0} A^{k}_{f,p\exp(tv)} .
	\end{equation}
\end{remark}
	
\begin{remark}\label{rem617167f4}
	We denote by $\Ad_p$ the differential at $e_G$ of the map $q\mapsto pqp^{-1}$.
	For every $f\in C^\infty(G;W)$, $p\in G$, $x\in\frk g$ and $k\in\N$,
	we have
	\begin{equation}\label{eq617721ba}
	A^{k}_{x^\dagger f,p} = \left( \Ad_p^{-1}(x)\rcontr  A^{\le k+s}_{f,p} \right)^{k}.
	\end{equation}
	
	Indeed, we have that the curves $t\mapsto p^{-1}\exp(tx)p$ and $t\mapsto\exp(t\Ad_p^{-1}(x))$ have equal derivative at $t=0$.
	Therefore,
	using the formula~\eqref{eq05162009}, which is still valid for non-horizontal vector fields, and~\eqref{eq6177b5e3}, we get that for every $v_1,\dots,v_k\in V_1$ that
	\begin{align*}
	&A^k_{x^\dagger f,p}[v_1,\dots,v_k]
	= \tilde v_k\cdots\tilde v_1 (x^\dagger f)_p(e) \\
	&= \left.\frac{\dd}{\dd s_1}\right|_{s_1=0} \cdots \left.\frac{\dd}{\dd s_k}\right|_{s_k=0} 
	(x^\dagger f)(p\exp(s_kv_k)\cdots\exp(s_1v_1)) \\
	&= \left.\frac{\dd}{\dd s_1}\right|_{s_1=0} \cdots \left.\frac{\dd}{\dd s_k}\right|_{s_k=0} 
	\left.\frac{\dd}{\dd t}\right|_{t=0}
	f(\exp(tx)p\exp(s_kv_k)\cdots\exp(s_1v_1)) \\
	&= \left.\frac{\dd}{\dd s_1}\right|_{s_1=0} \cdots \left.\frac{\dd}{\dd s_k}\right|_{s_k=0} 
	\left.\frac{\dd}{\dd t}\right|_{t=0}
	f(pp^{-1}\exp(tx)p\exp(s_kv_k)\cdots\exp(s_1v_1)) \\
	&= \left.\frac{\dd}{\dd s_1}\right|_{s_1=0} \cdots \left.\frac{\dd}{\dd s_k}\right|_{s_k=0} 
	\left.\frac{\dd}{\dd t}\right|_{t=0}
	f(p\exp(t\Ad_p^{-1}(x))\exp(s_kv_k)\cdots\exp(s_1v_1)) \\
	&= \left.\frac{\dd}{\dd t}\right|_{t=0} A^{k}_{f,p\exp(t\Ad_p^{-1}(x))}(v_1,\dots,v_k) \\
	&= \left( \Ad_p^{-1}(x)\rcontr  A^{\le k+s}_{f,p} \right)^{k}(v_1,\dots,v_k) ,
	\end{align*}
	that is, we get~\eqref{eq617721ba}.
\end{remark}

%%%%%%%%%%%%%%%%%%%%%%%%%%%%%%%%%%%%%%%%%%%%%%%%%%%%%%%%%%%%%%%
\subsection{The Lie algebra of the Jet space}\label{sec05301830}
By Proposition~\ref{prop05301400}, the map $v\mapsto v\rcontr$ is a Lie algebra representation of $\frk g$ on $\HorDer^{\le m}(\frk g;W)$.
Applying the results listed in Section~\ref{app05301812} to this representation, we obtain the Lie algebra
\[
\jet^m(\frk g;W) := \frk g \ltimes \HorDer^{\le m}(\frk g;W) ,
\]
where
\begin{equation}\label{eq5ebb173d}
[(v,A),(w,B)] = ([v,w],w\rcontr A - v\rcontr B) .
\end{equation}
See Section~\ref{sec62d95272} for a detailed example.

\begin{lemma}\label{lem6177b00e}
	$\jet^m(\frk g;W)$ is a stratified Lie algebra of step $\max\{s,m+1\}$, where $s$ is the step of $\frk g$.
	The $k$-th layer of $\jet^m(\frk g;W)$ is
	\[
	\jet^m(\frk g;W)_k := V_k\times \HorDer^{m+1-k}(\frk g;W)
	\]
	where $V_k=\{0\}$ if $k>s$ and $\HorDer^{m+1-k}(\frk g;W)=\{0\}$ if $k> m+1$.
\end{lemma}
\begin{proof}
	One easily sees that 
	\[
	[\jet^m(\frk g;W)_1,\jet^m(\frk g;W)_k] \subset \jet^m(\frk g;W)_{k+1} .
	\]
	Moreover, in the previous equation,
	the linear span of the left-hand side is equal to the right-hand side because of Lemma~\ref{lem5ead3862}.
\end{proof}

%%%%%%%%%%%%%%%%%%%%%%%%%%%%%%%%%%%%%%%%%%%%%%%%%%%%%%%%%%%%%%%
\subsection{The Jet space}\label{sec5ebc3ffe}
The simply connected Lie group corresponding to the Lie algebra $\jet^m(\frk g;W)$ is the semidirect product $\Jet^m(\frk g;W) = G \ltimes \HorDer^{\le m}(\frk g;W)$
with group operation given by Section~\ref{app05301812} as
\[
(\exp(v),A)(\exp(w),B) = ( \exp(v)\exp(w) , B+e^{(w\rcontr)}A ) .
\]
Using the inverse map $\log:G\to\frk g$ of the exponential map, we may also express the multiplication directly on $G$ by 
\[
(a,A)(b,B) = ( ab , B+e^{(\log(b)\rcontr)}A ) .
\]
The group $\Jet^m(\frk g;W)$ will be called the \emph{$W$-valued jet space over $G$ of order $m$}. 

Since $\Jet^m(G;W)$ is the product of $G$ and the vector space $\HorDer^{\le m}(\frk g;W)$, it is a smooth vector bundle over $G$, where for each $(a,A)\in \Jet^m(G;W)$ we have the following canonical identifications for the tangent and cotangent spaces:
\begin{equation}\label{eq61778ecd}
\begin{aligned}
T_{(a,A)}\Jet^m(G;W) &\simeq T_aG \times \HorDer^{\le m}(\frk g;W) , \\
T_{(a,A)}^*\Jet^m(G;W) &\simeq T_a^*G \times \HorDer^{\le m}(\frk g;W)^* .
\end{aligned}
\end{equation}

For $(x,X) \in \jet^m(\frk g;W)$ and $(a,A) \in \Jet^m(G;W)$, the differential of the left translation $L_{(a,A)}$ at $(0,0)$ in the direction $(x,X)$ is given by 
\[
d L_{(a,A)}|_{(0,0)} \begin{pmatrix} x \\ X \end{pmatrix}
= 
\begin{pmatrix}
d L_{a} & 0 \\
\cdot\rcontr A & \Id
\end{pmatrix}
\begin{pmatrix} x \\ X \end{pmatrix}
=
\begin{pmatrix}
\tilde x(a)  \\
x\rcontr A + X
\end{pmatrix}.
\]
The left invariant vector field generated by $(x,X)$, denoted $ (x,X)^{\widetilde{~} }$, when evaluated at $(a,A) \in \Jet^m(G;W)$, has the following form:
\[
(x,X)^{\widetilde{~}} (a,A) = (\tilde x(a), x\rcontr A + X ) .
\]

Since $G$ is simply connected and nilpotent, the exponential map $\exp:\frk g\to G$ is a diffeomorphism and it is common to identify $\frk g$ with $G$ via $\exp$.
In other words, one can reconstruct $G$ on $\frk g$ using the BCH formula.
This identification is very useful.
However, notice that the exponential map $\frk g\times\HorDer^{\le m}(\frk g;W) = \jet^m(\frk g;W)\to\Jet^m(G;W) = G\times\HorDer^{\le m}(\frk g;W)$ is not the identity map, but it is described in~\eqref{eq6171635c}.

%%%%%%%%%%%%%%%%%%%%%%%%%%%%%%%%%%%%%%%%%%%%%%%%%%%%%%%%%%%%%%%
\subsection{The Cartan distribution}

For $0\le\ell<m$, we define a vector-valued form $\omega^\ell\in\Omega^1(\Jet^m(G;W);\HorDer^\ell(\frk g;W))$, where for each $(a,A)\in\Jet^m(G;W)$ and all $(x,X)\in T_{(a,A)}\Jet^m(G;W)$, the value of $\omega^\ell(a,A)$ is given by  
\[
\omega^\ell(a,A)(x,X) := X^\ell - (dL_a^{-1}x)_1\rcontr A^{\ell+1} ,
\]
where $(dL_a^{-1}x)_1=\Pi_1(dL_a^{-1}x)\in V_1$ is the projection of $dL_a^{-1}x$ to $V_1$.%
\footnote{In fact, we could define only one $\omega\in \Omega^1(\Jet^m(G;W);\HorDer^{\le m-1}(\frk g;W))$ as 
\[
\omega(a,A)(x,X) := X^{\le m-1} - (dL_a^{-1}x)\rcontr A .
\]
Although this is more pleasant for the eye, we find technically more useful to have the single forms $\omega^\ell$ explicitly named.
}

For every $j\in\{1,\dots,s\}$, define the $V_j$ valued forms $\theta^j\in\Omega^1(\Jet^m(G;W);V_j)$ by
\[
\theta^j(a,A)(x,X) = \Pi_j(dL_a^{-1}x) .
\]

\begin{lemma}\label{lem61779343}
	The forms $\omega^\ell$ and $\theta^j$ are left invariant under the group operation
	of  $\Jet^m(\frk g;W)$.
\end{lemma}
\begin{proof}
	We show that the forms $\omega^\ell$ and $\theta^j$ are constant on left-invariant vector fields. To this end, let $x\in\frk g$ and $X\in\HorDer^{\le m}(\frk g;W)$. For every $(a,A)$, we have
	\begin{align*}
		\omega^\ell(a,A) [(x,X)^{\widetilde{~}}(a,A)]
		&= \omega^\ell(a,A)[\tilde x(a), x\rcontr A + X] \\
		&= (x\rcontr A + X)^\ell - x\rcontr A^{\ell+1}  \\
		&= x\rcontr A^{\ell+1} + X^\ell - x\rcontr A^{\ell+1} \\
		&= X^\ell 
		= \omega^\ell(e,0)[(x,X)^{\widetilde{~}}(e,0)] .
	\end{align*}
	Therefore, $\omega^\ell$ is constant on left-invariant vector fields $(x,X)^{\widetilde{~}}$, and so $\omega^\ell$ is left invariant.
	
	Similarly, from 
	$%\[%\begin{align*}
	\theta^j(a,A)[(x,X)^{\widetilde{~}}(a,A)]
	= \theta^j(a,A)[\tilde x(a), x\rcontr A + X]
	= x_j
	$, %\]%\end{align*}
	it again follows that $\theta^j$ is constant on left-invariant vector fields.
\end{proof}

The \emph{Cartan distribution} or \emph{contact jet structure}\footnotemark{} on $\Jet^m(G;W)$
is the distribution
\[
\scr H^m  = \bigcap_{\ell=0}^{m-1}\ker\omega^\ell \cap \bigcap_{j=2}^s\ker\theta^j .
\]
\footnotetext{We call the subbundle $\scr H^m$ a ``contact structure'', although the term ``contact'' is also used in the literature to indicate more specific distributions. 
However, our use of the word ``contact'' is definitely not uncommon, in particular with reference to jet spaces: see for instance \cite{olver1995equivalence}.}%
In other words, if $(a,A)\in\Jet^m(G;W)$, the space $\scr H^m_{(a,A)}$  consists of points of the form $(\tilde v(a),B)$,  where $v\in V_1$ and $ B^\ell - v\rcontr A^{\ell+1} =0$ for all $0\le\ell<m-1$.

Notice that $\scr H^m$ is left-invariant by Lemma~\ref{lem61779343} and that $\scr H^m_{(e,0)}=\jet^m(\frk g;W)_1$ is the first layer of the stratified Lie algebra $\jet^m(\frk g;W)$.

Next, we want to describe a frame for $\scr H^m$.
If $v_1,\dots,v_r$ is a basis of $V_1$ and $B_1^m,\dots,B_N^m$ is a basis of $\HorDer^m(\frk g;W)$, then a basis of $\scr H^m_{(a,A)}$ is given by the vectors
\begin{equation}\label{jmhvects} 
\begin{aligned}
 {\mathbb{X}}_j(a,A) &:= (\tilde v_j(a),v_j\rcontr A) , 
 	 \quad j\in\{1,\dots,r\}, \\
 {\mathbb{Y}}_j(a,A) &:= (0,B_j^m), 
 	\quad \quad \quad  j\in\{1,\dots,N\} . 
\end{aligned}
\end{equation}
To compute the Lie brackets of the vector fields at \eqref{jmhvects}, we first consider 
vector fields of the form ${\mathbb{X}}_v(a,A)=(\tilde v(a),v\rcontr A)$, for $v\in V_1$.
From Lemma~\ref{lem05291925}, it follows that
\[
[{\mathbb{X}}_v,{\mathbb{X}}_w] (a,A) = \left(
	[v,w]^{\widetilde{~}} (a) , w\rcontr(v\rcontr A) - v\rcontr(w\rcontr A)
	\right) 
	= \left(
	[v,w]^{\widetilde{~}} (a) , [v,w]\rcontr A
	\right) ,
\]
which gives the Lie brackets $[{\mathbb{X}}_j,{\mathbb{X}}_k]$.
The remaining brackets are
\[
[{\mathbb{X}}_j,{\mathbb{Y}}_k] = (0,v_j\rcontr B^m_k) 
\quad\text{and}\quad
[{\mathbb{Y}}_j,{\mathbb{Y}}_k] = 0 .
\]

If $\pi_m:\Jet^{m+1}(G;W)\to \Jet^m(G;W)$ is the projection along $\HorDer^{m+1}(\frk g;W)$, then  $\Span \, d\pi_m(\scr H^{m+1}) = \scr H^m$, by Lemma~\ref{lem5ead3862} (cfr.~section~\ref{sec_proj}).
Moreover, if $\pi_G:\Jet^m(G;W)\to G$ is the projection onto the first factor, then  $d\pi_G(\scr H^m)=\scr H_G$.

%%%%%%%%%%%%%%%%%%%%%%%%%%%%%%%%%%%%%%%%%%%%%%%%%%%%%%%%%%%%%%%
\subsection{Characterization of Jets of functions}
Recall that $\Jet^m(G;W)\to G$ is a vector bundle. 
For a smooth function $f:G\to W$, the \emph{$m$-jet of $f$} is the section of $\Jet^m(G;W)$ defined by 
\[
\Jet^m f(p) = \big( p, A_{f,p}^{\le m} \big) .
\]

\begin{proposition}\label{prop5eafeb43}
	Let $\gamma:G\to\Jet^m(G;W)$ be a smooth section.
	The following statements are equivalent:
	\begin{enumerate}[label=(\roman*)]
	\item\label{item5ead4605}
	There exists $f:G\to W$ smooth such that $\gamma=\Jet^mf$;
	\item\label{item5ead4606}
	$\gamma^*\omega^\ell|_{\scr H_G}=0$ for every $0\le\ell\le m-1$;
	\item\label{item5ead4607}
	$d\gamma(\scr H_G)\subset\scr H^m$. 
	\end{enumerate}
\end{proposition}
\begin{proof}
	To begin, note that since $\gamma$ is assumed to be a section of $\Jet^m(G;W)$, then \ref{item5ead4606} and \ref{item5ead4607} are equivalent.
	So, we only need to prove \ref{item5ead4605}$\IFF$\ref{item5ead4606}.
	
	\ref{item5ead4605}$\THEN$\ref{item5ead4606}.
	Assume $\gamma=\Jet^mf$ and let $p\in G$ and $v\in V_1$.
	Computing the differential $d\gamma(\tilde v(p))$ we get:
	\begin{align*}
	d\gamma(\tilde v(p)) 
	&= \left.\frac{\dd}{\dd t}\right|_{t=0} \gamma(p\exp(tv)) \\
	&= \left.\frac{\dd}{\dd t}\right|_{t=0} \left(p\exp(tv),A_{f,p\exp(tv)}^{\le m} \right) \\
	&= \left(\tilde v(p) , v\rcontr A_{f,p}^{\le m+1} \right),
	\end{align*}
	where we applied~\eqref{eq05161949} in the last identity.
	It is clear that $d\gamma(\tilde v(p))$ is a linear combination of the vectors in~\eqref{jmhvects}.
		
	\ref{item5ead4606}$\THEN$\ref{item5ead4605}.
	We write $\gamma(p) = (p,\sum_{j=0}^m\gamma^j(p))$, where $\gamma^j(p)\in\HorDer^j(\frk g;W)$.
	Since $\gamma$ is smooth, the function $f=\gamma^0$ is a smooth function $G\to W$.
	We need to show that $\gamma=\Jet^mf$. 
	To this end we show by induction that the claim $\gamma^k=A_f^k$ is true for $k=0, \dots m$, starting with the fact that for $k=0$, the claim is true by definition, i.e., $\gamma^0=f$.
	
	Let $k<m$ and assume $\gamma^k=A_f^k$: we shall show $\gamma^{k+1}=A_f^{k+1}$. Since \ref{item5ead4606} holds, then for all $p\in G$ and $v\in V_1$, we have
	\[
	0 = \omega^k(\gamma(p))[d\gamma(p)[\tilde v(p)]]
	= \left.\frac{\dd}{\dd t}\right|_{t=0}\gamma^k(p\exp(tv)) - v\rcontr\gamma^{k+1}(p) .
	\]
	Let $p\in G$ and $v_1,\dots,v_k,v\in V_1$.
	Then, using the inductive hypothesis and~\eqref{eq05161949}, we get 
	\begin{align*}
	\gamma^{k+1}(p)(v_1,\dots,v_k,v)
	&= (v\rcontr\gamma^{k+1}(p))(v_1,\dots,v_k) \\
	&= \left(\left.\frac{\dd}{\dd t}\right|_{t=0}\gamma^k(p\exp(tv))\right) (v_1,\dots,v_k) \\
	&= \left(\left.\frac{\dd}{\dd t}\right|_{t=0}A_{f,p\exp(tv)}^k\right)(v_1,\dots,v_k) \\
	&= \left( v\rcontr A_{f,p}^{k+1}\right) (v_1,\dots,v_k) \\
	&=  A_{f,p}^{k+1}(v_1,\dots,v_k,v) .
	\end{align*}
	We conclude that $\gamma^{k+1}=A_f^{k+1}$ and the proof is complete.
\end{proof}

%%%%%%%%%%%%%%%%%%%%%%%%%%%%%%%%%%%%%%%%%%%%%%%%%%%%%%%%%%%%%%%
%%%%%%%%%%%%%%%%%%%%%%%%%%%%%%%%%%%%%%%%%%%%%%%%%%%%%%%%%%%%%%%
\section{Jets via Taylor Expansions}\label{sec616fbbc5}
In this section we build a link between our notion of horizontal jet of a function and the notion of weighted Taylor expansion developed in \cite{MR657581}.
We will gain both a second perspective on horizontal jets and a few existence results that are needed to develop a theory of horizontal jets.

Throughout the section, $G$ is a stratified Lie group and $W$ a finite-dimensional real vector space.

%%%%%%%%%%%%%%%%%%%%%%%%%%%%%%%%%%%%%%%%%%%%%%%%%%%%%%%%%%%%%%%
\subsection{Homogeneity on stratified Lie groups}
In this subsection we introduce polynomials and homogeneous differential operators. We provide a sketch of some well known results and concepts with the particular goal of proving the second assertion in Proposition~\ref{prop617a5357}, which links homogeneous operators to the spaces $\HorDer^m(\frk g;W)$ which were introduced in the previous section.
From the combination of Propositions~\ref{prop606df205} and~\ref{prop617a5357}, we obtain an identification of $\HorDer^m(\frk g;W)$ with the space of polynomials $\Poly^m_p(G;W)$, for each $p\in G$.

Dilations $\delta_\lambda:G\to G$ and left-translations $L_p:G\to G$
define linear operators $\delta_\lambda^*,L_p^*:C^\infty(G;W)\to C^\infty(G;W)$ via pre-composition, that is
\[
\delta_\lambda^*f = f\circ\delta_\lambda
\qquad\text{and}\qquad
L_p^*f = f\circ L_p .
\]
We define \emph{dilations centered at }$p\in G$ by
\begin{equation}\label{eq61dd9179}
\delta_{p,\lambda} = L_p\circ\delta_\lambda\circ L_p^{-1} .
\end{equation}
A function $f\in C^\infty(G;W)$ is a \emph{homogeneous polynomial of weighted degree $m\in\N$ at $p\in G$}  if 
$\delta_{p,\lambda}^*f=\lambda^m f$ for all $\lambda>0$.
We denote the vector space of all homogeneous polynomials of weighted degree $m\in\N$ at $p\in G$ with values in $W$
by $\Poly^m_p(G;W)$.
One can easily check that for every $p\in G$ and $m\in\N$, we have
\begin{equation}\label{eq61715bca}
L_g^*(\Poly_p^m(G;W)) = \Poly_{g^{-1}p}^m(G;W) .
\end{equation}

One can show that $\Poly_{p}^m(G;W)\subset \bigoplus_{k=0}^m \Poly_{q}^k(G;W)$ for every $p,q\in G$, see~\cite{MR657581} for details.
Therefore, $\Poly^{\le m}(G;W)=\bigoplus_{k=0}^m \Poly_{p}^k(G;W)$ does not depend on $p$.
Elements of $\Poly^{\le m}(G;W)$ are \emph{polynomial of weighted degree $m\in\N$} on $G$. 
Moreover, $\Poly_{p}^k(G;W) = \Poly_{p}^k(G;\R)\otimes W$.

A linear operator $E:C^\infty(G)\to C^\infty(G)$ is \emph{homogeneous of degree $k\in\Z$} if
\[
\delta_\lambda^* \circ E \circ (\delta_\lambda^*)^{-1} = \lambda^k E
\]
for all $\lambda>0$.

One can easily check that if $E_1$, $E_2$ are homogeneous operators of degree $k_1$ and $k_2$, respectively, then the composition $E_1E_2$ is homogeneous of degree $k_1+k_2$.
Moreover, if $f\in\Poly_{e}^m(G)$, then the multiplier operator $m_f(u)=fu$ is homogeneous of degree $m$.
If $v\in V_k$, then the differential operator $\tilde v$ is homogeneous of order $-k$.

Let $\UniEnvAlg(G)$ be the universal enveloping algebra of $G$, that is, the algebra of all left-invariant differential operators on $C^\infty(G)$.
Elements of $\UniEnvAlg(G)$ are linear combinations of operators of the type $\tilde X_m\cdots \tilde X_1$, where each $\tilde X_j$ is a left-invariant vector field on $G$.
For $m\in\N$, we denote by $\UniEnvAlg^m(G)$ the space of elements in $\UniEnvAlg(G)$ that are homogeneous of degree $-m$.
One can show that $ \UniEnvAlg(G)= \bigoplus_{m=1}^\infty\UniEnvAlg^m(G)$.
Moreover, $\UniEnvAlg^m(G)$ is the dual space of $\Poly_{e}^m(G)$, as shown in \cite[Prop.~1.30]{MR657581}.

\begin{proposition}[{\cite[Prop.~1.30]{MR657581}}]\label{prop606df205}
	For every $p\in G$, the pairing 
	\[
	\langle \cdot|\cdot \rangle_p:\UniEnvAlg^m(G)\times\Poly^m_p(G) \to \R,
	\qquad
	\langle D|f \rangle_p := Df(p) 
	\]
	defines a linear isomorphism $\psi_p$ from $\Poly^m_p(G)$ to the dual $\UniEnvAlg^m(G)^*$.
	Moreover, $\psi_p$ extends to an isomorphism from $\Poly^m_p(G;W)$ to $\UniEnvAlg^m(G)^*\otimes W$, for every finite-dimensional vector space $W$.
\end{proposition}

\begin{corollary}\label{cor606df345}
	If $f\in C^\infty(G;W)$, $p\in G$ and $m\in\N$, then there exists a unique $P^m_{f,p}\in\Poly^m_p(G;W)$ such that
	\[
	Df(p) = DP^m_{f,p}(p)
	\]
	for all $D\in\UniEnvAlg^m(G)$.
\end{corollary}

Corollary~\ref{cor606df345} follows directly from Proposition~\ref{prop606df205}, because $D\mapsto Df(p)$ defines an element of $(\UniEnvAlg^m(G))^*$.
The polynomial $\sum_{k=0}^m P^k_{f,p} \in \Poly ^{\le m}(G;W)$ is the \emph{homogeneous Taylor expansion of $f$ at $p$ of order $m$}.

\begin{proposition}\label{prop617a5357}
	For every $m\in\N$, the function
	\begin{equation*}%\label{eq61715343}
	\tau:\Tensor^m(V_1)\to \UniEnvAlg^m(G),
	\qquad
	\tau(v_1\otimes\dots\otimes v_m) = \tilde v_m\cdots \tilde v_1 ,
	\end{equation*}
	is surjective and its transpose maps $\UniEnvAlg^m(G)^*$ onto $\HorDer^m(\frk g;\R)\subset \Tensor^m(V_1^*)$,
	that is,
	\[
	\tau^*:\UniEnvAlg^m(G)^* \to \HorDer^m(\frk g;\R) 
	\]
	is a linear isomorphism.
	Moreover, $\tau^*$ extends to a linear isomorphism from $\UniEnvAlg^m(G)^*\otimes W$ to $\HorDer(\frk g;W)\subset\Tensor^m(V_1^*)\otimes W$.
\end{proposition}
\begin{proof}
	The fact that $\tau$ is surjective is standard and we don't prove it here.

	Since $\tau:\Tensor^m(V_1)\to \UniEnvAlg^m(G)$ is a surjective linear map,
its transpose $\tau^*$ is a linear embedding of $\UniEnvAlg^m(G)^*$ into $\Tensor^m(V_1^*)$.
	The proof that $\tau^*$ maps $\UniEnvAlg^m(G)^*$ onto $\HorDer^m(\frk g;\R)$ is divided into two parts.
	
	First, we show that $\tau^*(\UniEnvAlg^m(G)^*) \subset \HorDer(\frk g;\R)$.
	Let $\alpha\in\UniEnvAlg^m(G)^*$.
	By Proposition~\ref{prop606df205}, there is $f\in\Poly^m_{e}(G)$ such that $\alpha(D) = Df(e)$ for all $D\in\UniEnvAlg^m(G)$.
	Then 
	\begin{multline*}
	\tau^*(\alpha)(v_1\otimes \dots\otimes v_m) 
	= \alpha(\tau(v_1\otimes \dots\otimes v_m)) \\
	= \tilde v_m\cdots \tilde v_1 f(e)
	= A^m_{f,e}(v_1\otimes \dots\otimes v_m) .
	\end{multline*}
	Second, we show that $\tau^*(\UniEnvAlg^m(G)^*) \supset \HorDer^m(\frk g;\R)$.
	If $A\in \HorDer^m(\frk g;\R)$, then there is $f\in C^\infty(G)$ such that $A=A^m_{f,e}$.
	Then, the computation above shows that $\tau^*(\alpha)=A$, where $\alpha=\psi_p(P^m_{f,e})$. 
\end{proof}

%%%%%%%%%%%%%%%%%%%%%%%%%%%%%%%%%%%%%%%%%%%%%%%%%%%%%%%%%%%%%%%
\subsection{The polynomial jet space}

For $P\in \Poly^{\le m}(G;W)$ and $v\in\frk g$, define
\[
v\rcontr P = v^\dagger P  \in \Poly^{\le m}(G;W),
\]
where $v^\dagger$ denotes the right-invariant vector field on $G$ with $v^\dagger(e_G)=v$.

\begin{remark}
	One can easily show that,
	if $P\in\Poly^m_{e}(G;W)$ and $v\in V_k$, then $v\rcontr P$ is an element of $\Poly^{m-k}_{e}(G;W)$.
	Notice that, in the latter claim, $e$ cannot be substituted with any point $p\in G$.
	The reason is that, for $p\neq e$, the spaces $\Poly^m_p(G;W)$ depend on the choice of left over right translations in the definition of $\delta_{p,\lambda}$, see~\eqref{eq61dd9179}.
\end{remark}

As a direct consequence of the standard equality $[v^\dagger,w^\dagger] = - [v,w]^\dagger$ for $v,w$ in a Lie algebra $\frk g$, it follows that the map $v\mapsto v\rcontr\in\End(\Poly^{\le m}(G;W))$ is a Lie algebra anti-morphism of $\frk g$, i.e.,
\begin{equation}
 	v\rcontr w\rcontr - w\rcontr v\rcontr = -[v,w]\rcontr ,
\end{equation}
for every $v,w\in\frk g$.

Therefore, we can apply the construction presented in Section~\ref{app05301812} with $\psi(v) = v\rcontr$.
Define $\jet^m_\Poly(\frk g;W)$ as $\frk g\ltimes_\psi \Poly^{\le m}(G;W)$ with Lie brackets
\begin{equation}\label{eq61781ab0}
[ (v,P) , (w,Q) ] = ([v,w], w\rcontr P - v\rcontr Q) 
= ([v,w],w^\dagger P - v^\dagger Q) .
\end{equation}
In parallel to Lemma~\ref{lem6177b00e}, the Lie algebra $\jet^m_\Poly(\frk g;W)$ is stratified with layers
\[
\jet^m_\Poly(\frk g;W)_k = V_k\oplus \Poly^{m+1-k}_{e}(G;W) ,
\]
where $V_k=\{0\}$ if $k>s$ and $\Poly^{m+1-k}_{e}(G;W)=\{0\}$ if $k> m+1$.
Define $\Jet^m_\Poly(G;W) = G\ltimes \Poly^{\le m}(G;W)$
as the Lie group with operation
\[
(\exp(v),P)(\exp(w),Q) 
= ( \exp(v)\exp(w) , Q + e^{v^\dagger} P ) .
\]

\begin{theorem}\label{thm6171c9a8}
	\begin{enumerate}[leftmargin=*,label=(\roman*)]
	\item
	For every $p\in G$,
	the map $\sigma_p:\Poly^m_p(G;W) \to \HorDer^m(\frk g;W)$, $\sigma_p(f) = A^m_{f,p}$ is a linear isomorphism.\footnote{Notice that the space $\HorDer^{\le m}(\frk g;\R)$ does not depend on $p$, but the way it is identified with $\Poly^{\le m}_p(G)$ does.}
	\item
	If $p\in G$, $f\in\Poly^{\le m}(G;W)$ and $v\in\frk g$, then
%	the identity \eqref{eq6177b5e3} holds for all $k$
%	and
	\begin{equation}\label{eq6177d5f6}
		\sigma_p(v\rcontr f) = \Ad_p^{-1}(v)\rcontr \sigma_p(f) \in \HorDer^{\le m}(G;W) .
	\end{equation}
	\item
	The map $\sigma:\jet^m_\Poly(\frk g;W)\to \jet^m(\frk g;W)$, $\sigma(v,f)=(v,A^{\le m}_{f,e})$ is an isomorphism of stratified Lie algebras.
	\item
	The map $\Phi:\Jet^m_\Poly(\frk g;W)\to \Jet^m(\frk g;W)$, $\Phi(p,f) = (p, A^{\le m}_{f,e})$ is the isomorphism of Lie groups with $\Phi_*=\sigma$.
	\end{enumerate}
\end{theorem}
\begin{proof}
	The composition of the isomorphisms $\psi_p$ form Proposition~\ref{prop606df205}
	and $\tau^*$ from Proposition~\ref{prop617a5357} gives the linear isomorphism
	\begin{equation*}%\label{eq617a58b8}
	\sigma_p:
	\Poly ^{m}_p(G;W) \overset{\psi_p}\longrightarrow 
	\UniEnvAlg^m(G;W)^* \overset{\tau^*}\longrightarrow
	\HorDer^m(\frk g;W),
	\qquad \sigma_p(f) = A^m_{f,p},
	\end{equation*}
	for each $p\in G$.
	
	The identity~\eqref{eq6177d5f6} follows from Remark~\ref{rem617167f4}.
	Notice in particular that $\sigma_{e}(x\rcontr f) = x\rcontr \sigma_{e}(f)$.
	Thus, $\sigma$ is an isomorphism of stratified Lie algebras.
	Finally, the last statement is proved by checking that
	\[
	\Phi(\exp_{\Jet_\Poly}(x,P)) = \exp_{\Jet}(\sigma(x,P))
	\]
	for every $x\in\frk g$ and $P\in\Poly^{\le m}(G)$,
	which is a computation with~\eqref{eq6171635c}.
\end{proof}

An immediate consequence of Theorem~\ref{thm6171c9a8} is the following statement.
\begin{corollary}\label{cor5eafe914}
		For every $p\in G$ and $A\in\HorDer^{\le m}(\frk g;W)$ there is $f\in C^\infty(G;W)$ such that $A_{f,p}^{\le m} = A$ and $A_{f,p}^k=0$ for all $k>m$.
		In particular, if $f\in\Poly^m_p(G;W)$, then $A^k_{f,p}=0$ for $k\neq m$.
\end{corollary}

%%%%%%%%%%%%%%%%%%%%%%%%%%%%%%%%%%%%%%%%%%%%%%%%%%%%%%%%%%%%%%%
\subsection{Bases}\label{sec61715b90}
Let $\scr B = (b_1,\dots,b_n)$ be a basis of $\frk g$ adapted to the stratification $\bigoplus_jV_j$, i.e., 
there is a non-decreasing sequence of integers $\{\w_i\}_{i=1}^n$ such that $b_i\in V_{\w_i}$ for all $i$.
If $I\in\N^n$ is a multi-index, define $\w(I)=\sum_{j=1}^n \w_jI_j$.
We denote by $x_j$ the exponential coordinates given by $\scr B$, i.e., smooth functions $G\to\R$ such that $\exp(\sum_jx_j(p)b_j)=p$ for all $p\in G$.
The \emph{homogeneous degree} of a monomial $x^I = \prod_{j=1}^n x_j^{I_j}$ is 
$\deg(x^I) = \w(I)$.

\begin{lemma}
	Using the above notation, $\{x^I\}_{\w(I)=m}$ is a basis of $\Poly^m_{e}(G;\R)$.
\end{lemma}
\begin{proof}
	It is clear that $\{x^I\}_{\w(I)=m}$ is a linearly independent subset of $\Poly^m_{e}(G;\R)$.
	To prove the claim, we set $|I|=\sum_{j=1}^n I_j$, so that $|I|\le \w(I) \le s|I|$ because $1\le w_i\le s$ for all $i$.
	Given $f\in\Poly^m_{e}(G;\R)$, we write the Taylor expansion of $f$  at 0 in exponential coordinates up to order $m$, that is,
	$f(x) = \sum_{|I|\le m} f_I x^I + \rho(x)$, 
	where $\rho\in C^\infty(G)$ is such that $|\rho(x)| \le C(|x|^{m+1})$ for $|x|\le 1$ and some $C\ge0$.
	Here $|\cdot|$ denotes the Euclidean norm.
	We rearrange this sum as
	\[%\begin{align*}
	f(x)
	= \sum_{|I|\le m} f_I x^I + \rho(x)
	= \sum_{k=0}^m \sum_{\w(I)=k} f_I x^I + \sum_{k=m+1}^{sm} \sum_{\substack{\w(I)=k\\|I|\le m}}f_I x^I+ \rho(x) .
	\]%\end{align*}
	Therefore, for $x$ fixed and $\lambda>0$, we have
	\[
	\frac{f(\delta_\lambda x)}{\lambda^m} - \sum_{\w(I)=m} f_I x^I
	= \sum_{k=0}^{m-1} \lambda^{k-m} \sum_{\w(I)=k} f_I x^I + \sum_{k=m+1}^{sm} \lambda^{k-m} \sum_{\substack{\w(I)=k\\|I|\le m}}f_I x^I+ \frac{\rho(\delta_\lambda x)}{\lambda^m} .
	\]
	Notice that the left-hand side of the latter identity is constant in $\lambda$, while on the right-hand side we have
	\begin{align*}
	\lim_{\lambda\to 0^+} |\sum_{k=0}^{m-1} \lambda^{k-m} \sum_{\w(I)=k} f_I x^I| &= \infty,
	& \text{ if $\sum_{\w(I)=k} f_I x^I\neq0$ for some $k$,} \\
	\lim_{\lambda\to 0^+} \sum_{k=m+1}^{sm} \lambda^{k-m} \sum_{\substack{\w(I)=k\\|I|\le m}}f_I x^I &= 0 ,\\
	\lim_{\lambda\to 0^+} \frac{\rho(\delta_\lambda x)}{\lambda^m} &= 0 ,
	\end{align*}
	where the last limit is follows from the fact $|\delta_\lambda x|\le \lambda|x|$.
	We thus obtain that $f(x) =\sum_{\w(I)=m} f_I x^I $.
\end{proof}

For a multi-index $I\in\N^n$, define 
$\tilde b^I = \tilde b_1^{I_1}\cdots \tilde b_n^{I_n}\in\UniEnvAlg(G)$.
Notice that we obtain $(\delta_\lambda)_*\tilde b^I = \lambda^{-\w(I)}\tilde b^I$. 
By the Poincaré--Birkhoff--Witt Theorem \cite[I.2.7]{MR0132805},
$\{\tilde b^I\}_{I\in\N^n}$ is a basis of $\UniEnvAlg(G)$.
Since $\tilde b^I\in\UniEnvAlg^m(G)$ if and only if $\w(I)=m$,
then $\{\tilde b^I:\w(I)= m\}$ is a basis of $\UniEnvAlg^m(G)$.

We have two ways to build a basis of $\HorDer^m(\frk g;W)$:
we can take $\{\sigma_{e}(x^I)\}_{\w(I)=m}$,
or the basis dual to $\{\tilde b^I\}_{\w(I)= m}$.
These are not the same, because it is \emph{not} true that $\langle \tilde b^I | x^J \rangle = 0$ if and only if $I=J$.

We choose to describe the second basis, that is, the basis dual to $\{\tilde b^I\}_{\w(I)= m}$, for both $\HorDer^m(\frk g;\R)$ and $\Poly_{e}^m(G;\R)$.
In other words, we will get $\{A_I\}_{\w(I)= m}\subset\HorDer^m(\frk g;\R)$ such that, for every $f\in C^\infty(G)$ and $p\in G$,
\begin{equation}
A^{m}_{f,p}  = \sum_{\w(I) = m} (\tilde b^If)(p)\cdot A_I  .
\end{equation}
And, for each $p\in G$, we will get $\{P_{p,I}\}_{\w(I)= m}\subset\Poly_p^m(G;\R)$ such that
\begin{equation}\label{eq617b9d84}
P^{m}_{f,p} = \sum_{\w(I) = m} (\tilde b^If)(p)\cdot P_{p,I} .
\end{equation}

We can obtain the polynomial $P_{p,I}$ by imposing $ \tilde b^IP^J(p) = \delta^I_J$ and then compute $A_I = A^m_{P_{e,I},e}$.
The identity~\eqref{eq617b9d84} gives us the homogeneous taylor expansion of $f\in C^\infty(G)$ at $p$ as
\[
f \sim \sum_{m\ge 0} \sum_{\w(I) = m} (\tilde b^If)(p)\cdot P_{p,I} .
\]

However, we can compute $A_I$ without taking derivatives and exploiting the algebra structure of $\UniEnvAlg(G)$.
Indeed, compute the functions $\tau_I:\Tensor^m(V_1)\to\R$ given by
\begin{equation}\label{eq617b9c61}
\tau(\xi) = \sum_{\w(I)= m} \tau_I(\xi) \tilde b^I,
\end{equation}
for $\xi\in \Tensor^m(V_1)$.
If $\{P_{p,I}\}_{\w(I)= m}$ is the basis of $\Poly^{m}_p(G)$ dual to $\{\tilde b^I\}_{\w(I)= m}$,
then 
\[
\langle \tau^*P_{p,I} | \xi \rangle
= \langle P_{p,I} | \tau(\xi) \rangle
= \tau_I(\xi) .
\]
Therefore, $A_I$ is given by 
\begin{equation}\label{eq6178f2b6}
\langle A_I|\xi \rangle = \tau_I(\xi) 
\end{equation}
for all $\xi \in \Tensor^m(V_1)$.
Notice that $A_I$ does not depend on $p$, while $P_{p,I}$ does.
Moreover, we can compute $A_I$ without computing $P_{p,I}$.
Clearly, the isomorphism $\sigma_p$ is given by $\sigma_p(P_{p,I}) = A_I$.

\begin{remark}\label{rem6093eb23}
	We now summarize an algorithm to compute
	a basis 
	for $\HorDer^m(\frk g;W)$:
	\begin{enumerate}
		\item
		Choose a basis $\scr B = (b_1,\dots,b_n)$ of $\frk g$ adapted to the stratification $\bigoplus_{j=1}^s V_j$;
		We denote by $\Xi^m$ the basis of $\Tensor^m(V_1)$ induced by $\scr B$.
		\item
		Compute the list of multi-indices $\scr I^m = \{I\in\N^n:\w(I)=m\}$.
		\item
		For each $\xi\in\Xi^m$ write $\tau(\xi)$ in the basis $\{\tilde b^I\}_{I\in\scr I^m}$ as in~\eqref{eq617b9c61}, 
		so that we obtain $\{\tau_I(\xi)\}_{I\in\scr I^m,\xi\in\Xi^m}$.
		\item
		For each $I\in\scr I^m$, compute $A_I = \sum_{\xi\in\Xi^m} \tau_I(\xi) \xi^*$,
		where $\xi^*$ is the element of $\Tensor^m(V_1^*)$ dual to $\xi$ with respect to the basis $\Xi^m$.
		\item
		The resulting $\{A_I\}_{I\in\scr I^m}$ is a basis of $\HorDer(\frk g;\R)$.
		Since $\HorDer^m(\frk g;W)$ is the tensor product $\HorDer^m(\frk g;\R)\otimes W$, a basis for $\HorDer^m(\frk g;W)$ is $\{A_I\otimes c\}_{I\in\scr I^m, c\in \scr C}$, for a basis $\scr C$ of $W$.
	\end{enumerate}
	See section~\ref{sec617bad0c} for an application.
\end{remark}

%%%%%%%%%%%%%%%%%%%%%%%%%%%%%%%%%%%%%%%%%%%%%%%%%%%%%%%%%%%%%%%
%%%%%%%%%%%%%%%%%%%%%%%%%%%%%%%%%%%%%%%%%%%%%%%%%%%%%%%%%%%%%%%
\section{Prolongation of Contact Maps}\label{sec610264b1}

%%%%%%%%%%%%%%%%%%%%%%%%%%%%%%%%%%%%%%%%%%%%%%%%%%%%%%%%%%%%%%%
\subsection{Projection to lower order jet bundles}\label{sec_proj}
In this section, we will work on $\Jet^m(G;W)$ and $\Jet^{m+1}(G;W)$ simultaneously, and so we will accent notation with~$\hat \,$ to indicate objects on $\Jet^{m+1}(G;W)$.
For example, we write $p$ for a point in $\Jet^{m}(G;W)$ and $\hat p$ for a point in $\Jet^{m+1}(G;W)$; we use $\omega^\ell$ for a contact form on $\Jet^m(G;W)$ and  $\hat \omega^\ell$ for a contact form on $\Jet^{m+1}(G;W)$, etc.

We denote by $\pi_m:\Jet^{m+1}(G;W)\to \Jet^{m}(G;W)$ the projection along $\HorDer^{m+1}(\frk g;W)$.
This map is not a morphism of Lie groups, because $\HorDer^{m+1}(\frk g;W)$ is not an ideal of $\jet^{m+1}(\frk g;W)$, but it maps the horizontal distribution onto the horizontal distribution, although not point by point.
In other words, if $\hat p \in \pi_m^{-1}(p)$ then $d\pi_m|_{\hat p} (\scr H^{m+1}_{\hat p})\neq\scr H^{m}_{p}$ (as one can see from~\eqref{eq6177ee74} below), however $\scr H^{m}_{p}$ is the closure of the union of $d\pi_m|_{\hat p}(\scr H^{m+1}_{\hat p})$ for all $\hat p \in \pi_m^{-1}(p)$.
In fact, we don't need to consider the whole fiber $\pi_m^{-1}(p)$, 
but any open subset already spans the horizontal distribution in the image, 
as the next lemma shows. 

%\begin{lemma}\label{lem5ebc442a}
%	\footnote{This lemma is not used anywhere, I think.}
%	If $p\in\Jet^{m}(G;W)$ and $g\in\Jet^{m-1}(G;W)$, then
%	\[
%	g\pi_m(p) = \pi_m(gp).
%	\]
%\end{lemma}
%\begin{proof}
%	Set $g=(a,A^{\le m-1})$ and $p=(b,B^{\le m})$.
%	On the one hand, we have
%	\[
%	g\pi_m(p)
%	= (a,A^{\le m-1})(b,B^{\le m-1})
%	= \left( ab, B^{\le m-1} + e^{(\log(b)\rcontr)}A^{m-1} \right).
%	\]
%	On the other hand, we have
%	\[
%	\pi_m(gp)
%	= \pi_m\left(  ab, B^{\le m} + e^{(\log(b)\rcontr)}A^{m-1} \right)
%	= \left(  ab, B^{\le m-1} + e^{(\log(b)\rcontr)}A^{m-1} \right).
%	\]
%\end{proof}

\begin{lemma}\label{lem5ebc4d52}
	If $p\in  \Jet^{m}(G;W)$ and $U \subset  \Jet^{m+1}(G;W)$ is open and satisfies $\pi_m^{-1}(p) \cap U \ne \emptyset$, then
	\begin{equation}\label{eq5ebc46a2}
	\scr H^{m}_{ p} = \Span\left( \bigcup 
		\left\{ d\pi_m|_{\hat p}(\scr H^{m+1}_{\hat p}) : \hat p \in \pi_m^{-1}(p) \cap U \right\}
		\right).
	\end{equation}
\end{lemma}
\begin{proof}
	\ic
	Let $\hat p=(a,A^{\le m+1}) \in \pi_m^{-1}(p) \cap U$. Since elements of $\scr H^{m+1}_{\hat p}$ have the form
	$(\tilde v(a),v\rcontr A^{\le m+1}+B^{m+1})$ where $v\in V_1$ and $B^{m+1}\in\HorDer^{m+1}(\frk g;W)$, it follows that
	\begin{equation}\label{eq6177ee74}
	d\pi_m|_{\hat p}\big ( (\tilde v(a),v\rcontr A^{\le m+1}+B^{m+1}) \big ) 
	= (\tilde v(a),v\rcontr A^{\le m+1}) \in \scr H^{m}_{p}.
	\end{equation}
	Since $\hat p \in \pi_m^{-1}(p) \cap U$ is arbitrary we can conclude that the right-hand side of $\eqref{eq5ebc46a2}$ is  a subset of $\scr H^{m}_{p}$.

	\ci 
	By Lemma~\ref{lem5ead3862}, if $\{v_1,\dots,v_r \} \subset V_1$ is a basis for $V_1$ and $\{B_1,\dots ,B_N\} \subset \HorDer^{m+1}(\frk g;W)$  is basis for  $\HorDer^{m+1}(\frk g;W)$, then $\{ v_i\rcontr B_j : i=1,\dots,r, j=1,\dots,N\}$ contains a basis for $\HorDer^{m}(\frk g;W)$. 
	Given $\hat p=(a,A^{\le m+1}) \in \pi_m^{-1}(p) \cap U$, define $\hat q_j \in \Jet^{m+1}(\frk g;W)$ by $\hat q_j=(a,A^{\le m+1} + B_j)$. 
	It follows that $\hat q_j \in \pi_m^{-1}(p)$ and by scaling if necessary, we can assume that each $B_j$ is close enough to $0$ so that $\hat q_j\in U$. 

	The left invariant field generated by $(v_i,0)$ evaluated at $\hat q_j$ is $$ \big(\tilde v_i (a), v_i \rcontr(A^{\le m+1} + B_j) \big )$$ and by definition belongs to $\scr H_{\hat q_i}^{m+1}$. Furthermore, 
	\[
	d\pi_m|_{\hat q_j}\big ( (\tilde v_j(a),v_j\rcontr (A^{\le m+1}+B_j)) \big )
	= (\tilde v_j(a),v_j\rcontr A^{\le m+1} + v_j\rcontr B_j) \in H,
	\] 
	where $H$ denotes the right hand side of \eqref{eq5ebc46a2}. Since $(\tilde v_j(a),v_j\rcontr A^{\le m+1}) \in H$ for all $j$, it follows that $(0,v_j\rcontr B_j)\in H$ for all $j$ and we conclude that $\HorDer^{m}(\frk g;W)\subset H$.
\end{proof}

%%%%%%%%%%%%%%%%%%%%%%%%%%%%%%%%%%%%%%%%%%%%%%%%%%%%%%%%%%%%%%%
\subsection{Contact maps}
Given a smooth map $F:\Omega\to\Jet^m(G;W)$ from an open set $\Omega\subset\Jet^m(G;W)$,
%onto its image $F(\Omega)\subset\Jet^m(G;W)$, 
we write it component wise as
\[
F(a,A) = \left(F_G(a,A) , F^0(a,A), \dots ,F^m(a,A) \right)
\]
where $F_G:\Omega\to G$ and $F^k:\Omega\to \HorDer^k(\frk g;W)$ are smooth maps. 
Furthermore, $F$ is said to be a \emph{contact map} if $d F(\scr H^m)\subset\scr H^m$ which in the case $m \geq 1$, is characterised by the following conditions:
	\begin{equation}\label{eq5eafe676}
	\begin{array}{rll}
	\omega^\ell(F(p)) \big (d F(p) {\mathbb{X}}_j(p) \big ) = 0, && 0\le\ell\le m-1,\ \forall j , \\
	\omega^\ell(F(p))\big (d F(p) {\mathbb{Y}}_j(p)\big ) = 0, && 0\le\ell\le m-1,\ \forall j , \\
	\theta^k(F(p))\big (d F(p) {\mathbb{X}}_j(p)\big ) = 0, && 2\le k \le s,\ \forall j , \\
	\theta^k(F(p))\big (d F(p) {\mathbb{Y}}_j(p)\big ) = 0, && 2\le k \le s,\ \forall j , \\
	\end{array}
	\end{equation}
	for all $p\in\Omega$. 
	In the case $m=0$, the conditions reduce to those given by the forms $\theta^k$ and simply mean that $F$ is contact map of some open set $\Omega \subseteq G\times W$, where 
	$\scr H^0_{(e_G,0)} = \jet^0(\frk g;W) = V_1\times W$.  
Note that in the case $G=\R^n$, the definition is consistent with the usual contact system as defined in \cite[chapter 4]{olver1995equivalence}.

%%%%%%%%%%%%%%%%%%%%%%%%%%%%%%%%%%%%%%%%%%%%%%%%%%%%%%%%%%%%%%%
\subsection{Prolongation} 
Suppose that $\Omega\subset\Jet^m(G;W)$ is open, and that $F:\Omega\to\Jet^m(G;W)$ is a contact map. 
If $\hat \Omega \subset \Jet^{m+1}(G;W)$ is open and satisfies $\pi_m( \hat \Omega) \subseteq \Omega$, then a smooth map $\hat F: \hat \Omega \to \Jet^{m+1}(G;W)$ satisfying
\begin{equation}\label{pimhatFFpim}
\pi_m \circ \hat F = F \circ \pi_m
\end{equation}
is called a \emph{prolongation of $F$} if $\hat F$ is a contact map. 
Since $\hat \omega^\ell =\pi_m^* \omega^\ell$, $\ell=1,\dots,m-1$, and $\hat \theta^k =\pi_m^* \theta^k$, $k=2,\dots,s$, any map $\hat F$ that satisfies condition \eqref{pimhatFFpim}, immediatley satisfies all the contact conditions~\eqref{eq5eafe676} on $\Jet^{m+1}(G;W)$ except those corresponding to $ \hat \omega^m$. 
In particular $\hat F_G = F_G \circ \pi_m$,  $\hat F^i = F^i \circ \pi_m$ for $i=0,\dots,m$ and $\hat F^{m+1}$ is determined by the contact conditions corresponding to $\hat \omega^m$. More precisely, the conditions
\begin{equation} \label{trivyj}
\hat \omega^m(\hat F(\hat p)) (d \hat F(\hat p) {\mathbb{\hat Y}}_j(\hat p)) = 0
\end{equation}
are trivial since from \eqref{pimhatFFpim} we have
\[
	d\pi_m (\hat F(\hat p))  d\hat F(\hat p)  {\mathbb{\hat Y}}_j(\hat p) = dF(p) d\pi_m(\hat p)   {\mathbb{\hat Y}}_j(\hat p)= dF(\pi_m (\hat p))0=0
\]
which implies that 
\[
	d\hat F(\hat p)  {\mathbb{\hat Y}}_j(\hat p) \in \text{ker} d\pi_m (\hat F(\hat p)) 
\]
or equivalently 
\begin{equation}\label{eq617657d3}
	d\hat F(\hat p)  {\mathbb{\hat Y}}_j(\hat p) 
	\in \Span\{ {\mathbb{\hat Y}}_k(\hat p) : k = 1,\dots \}
\end{equation}
giving \eqref{trivyj}.

It follows that the determining conditions for $\hat F^{m+1}$ are the equations
\begin{equation}\label{prolongmtomplusone}
\hat \omega^m(\hat F(\hat p))( d \hat F(\hat p)  {\mathbb{\hat X}}_j(\hat p)) = 0 \quad \forall j. 
\end{equation} 

The existence of prolongations is governed by the following result.

\begin{theorem}[Prolongation Theorem]\label{thm5eafe9c5}
	Supppose $m\ge0$, $\Omega\subset\Jet^m(G;W)$ is open, and that $F:\Omega\to\Jet^m(G;W)$ is a contact map.
	Let $v_1,\dots,v_r$ be a basis of $V_1$ and, for $j\in\{1,\dots,r\}$, define
	$\tilde N_j:\Jet^{m+1}(G;W) \to TG$ as 
	\[
	\tilde N_j (\hat p) 
	= dF_G({\pi_m(\hat p)}) d\pi_m(\hat p) \hat{\bb X}_j(\hat p). 
	\]
	Define $\hat \Omega\subset\Jet^{m+1}(G;W)$ as the open set where $\tilde N_1,\dots,\tilde N_r$ are pointwise linearly independent.
	Then there is a unique contact map $\hat F:\hat \Omega\to\Jet^{m+1}(G;W)$ such that 
	\[
	\pi_m\circ \hat F = F\circ\pi_m.
	\]
	Moreover, if $F$ is a diffeomorphism, then $\hat F$ is also a diffeomorphism and $\hat F^{-1}$ is the prolongation of $F^{-1}$.
\end{theorem}

\begin{remark}
	Uniqueness of the prolongation $\hat F$ is proved only on $\hat\Omega$ (cfr.~\eqref{eq5eafde7c}).
	For instance, a constant map $F$ is contact and admits infinite prolongations, but $\hat\Omega=\emptyset$.
\end{remark}

\begin{remark}\label{rem61765cb1}
	We don't know the size of $\hat\Omega$ in general.
	{\it If $F$ is itself a prolongation of a contact map on $\Jet^{m-1}(G;W)$, 
	then $\hat\Omega = \pi_m^{-1}(\Omega)$.}
	Notice that, by the results in the following Section~\ref{sec616fa813}, $F$ is a prolongation as soon as $m\ge2$.
	
	Let's prove the above statement about $\hat\Omega$.
	First, notice that, if $\hat p = (a,A^{\le m+1})$ and $p=\pi_m(\hat p)$, then
	\[
	d\pi_m(\hat p) \hat{\bb X}_j(\hat p)
	= (\tilde v_j(a) , v_j\rcontr A^{\le m} + v_j\rcontr A^{m+1}) 
	= \bb X_j(p) +  v_j\rcontr A^{m+1} ,
	\]
	where $ v_j\rcontr A^{m+1}\in \HorDer^{m}(\frk g;W)$.
	
	Second, if $ \pi_G:\Jet^{m}(G;W)\to G$ is the projection onto $G$, then
	the restriction $d\pi_G(q)|_{\scr H^{m}} :\scr H^{m}(q) \to \scr H_G(\pi_G(q))$ is surjective with kernel $\HorDer^{m}(\frk g;W)$, for every $q\in \Jet^{m}(G;W)$.
	
	Third, if $F$ is a prolongation, then it satisfies~\eqref{eq617657d3} (without hats).
	In particular, the span of the vectors $dF(p) \bb X_j(p)$ is transversal to $\HorDer^{m}(\frk g;W)$.
	Hence, $d\pi_G(F(p))$ is an isomorphism between the span of the vectors $dF(p) \bb X_j(p)$ and $\scr H_G(\pi_G(q))$.	
	
	Thus, 
	\[
	\tilde N_j (\hat p) 
	= dF_G({\pi_m(\hat p)}) d\pi_m(\hat p) \hat{\bb X}_j(\hat p)
	= d\pi_G(F(p)) dF(p) \bb X_j(p) 
	\]
	are linearly independent.
\end{remark}

For the proof of Theorem~\ref{thm5eafe9c5}, we need the following technical lemma.

\begin{lemma}\label{lem5eafdaf1}
	Assuming the hypothesis of Theorem~\ref{thm5eafe9c5},
	for every $f:G\to W$ smooth and $a\in G$ such that $\Jet^{m+1}f(a)\in\hat \Omega$, there exists a smooth function $h:G\to W$  and a neighborhood $U$ of $a$, such that 
	\begin{equation}\label{eq06011605}
	F\circ \Jet^mf|_{U} = \Jet^mh\circ F_G\circ\Jet^mf|_{U} .
	\end{equation}
	Moreover, $F_G\circ\Jet^mf(U)$ is an open subset of $G$ and $F_G\circ\Jet^mf$ is a diffeomorphism on $U$.
\end{lemma}
\begin{equation}\label{eq619de3b7}
\xymatrix{
\Jet^m(G;W)\ar[r]^{F}\ar[dr]_{F_G} & \Jet^m(G;W) \ar@<-.5ex>[d]_{\pi_G} \\
G\ar[u]^{\Jet^mf}\ar[r]_{F_G\circ\Jet^mf} & G\ar@{-->}@<-.5ex>[u]_{\Jet^mh}
}
\end{equation}
\begin{proof}
	First, notice that, if $\hat p = (a,A^{\le m+1})$ and $p=\pi_m(\hat p)=(a,A^{\le m})$, then
	\[
	d\pi_m(\hat p) \hat{\bb X}_j(\hat p)
	= (\tilde v_j(a) , v_j\rcontr A^{\le m} + v_j\rcontr A^{m+1}) 
	= \bb X_j(p) +  (0,v_j\rcontr A^{m+1}) \in\scr H^m_p,
	\]
	where $ v_j\rcontr A^{m+1}\in \HorDer^{m}(\frk g;W)$.
	Since $F$ and the projection $\pi_G$ to $G$ are contact, $F_G=\pi_G\circ F$ is also contact.
	Thus, $\tilde N_j$ takes values in $\scr H_G$.
	Define:
	\begin{equation}\label{eq5eafda06}
	N_j(\hat p) = d L_{F_G(p)}^{-1} \tilde N_j(\hat p)\in V_1 .
	\end{equation}
	
	Second, since $F$ is contact,  it follows that for all $0\le\ell\le m-1$ and all $\hat p \in \hat \Omega \subset \Jet^{m+1}(G;W)$, we have that
	\begin{align*}
	0 
	&= \omega^\ell(F(p))(d F(p)d\pi_m(\hat p) \hat{\bb X}_j(\hat p))\\
	&= d F^\ell(p) d\pi_m(\hat p) \hat{\bb X}_j(\hat p) 
		-  \left( d L_{F_G(p)}^{-1}   d F_G(p) d\pi_m(\hat p) \hat{\bb X}_j(\hat p)\right)_1 \rcontr F^{\ell+1}(p) \\
	&= d F^\ell(p) d\pi_m(\hat p) \hat{\bb X}_j(\hat p) 
		-  d L_{F_G(p)}^{-1}  \tilde N_j(\hat p) \rcontr F^{\ell+1}(p) \\
	&= d F^\ell(p) d\pi_m(\hat p) \hat{\bb X}_j(\hat p)
		- N_j( \hat p)\rcontr F^{\ell+1}(p) ,
	\end{align*}
	that is
	\begin{equation}\label{eq06011622}
	d F^\ell(p) \left ( d\pi_m(\hat p) \hat{\bb X}_j(\hat p) \right )
	= N_j(\hat p)\rcontr F^{\ell+1}(p).
	\end{equation}
	
	Next, let $f:G\to W$ be a smooth function.
	Notice that
	\begin{equation}\label{eq619de09a}
	\begin{aligned}
	d(F_G\circ\Jet^mf)(a) \tilde v_j(a)
	&= d F_G(\Jet^mf(a)) \left( \tilde v_j(a) , \left.\frac{\dd}{\dd t}\right|_{t=0} A^{\le m}_{f,a\exp(tv_j)} \right) \\
	&= d F_G(\Jet^mf(a)) \left( \tilde v_j(a) , v_j\rcontr A^{\le m+1}_{f,a}  \right)  \\
	&= \tilde N_j(\Jet^{m+1}f(a)) .
	\end{aligned}
	\end{equation}
	
	Since the vectors $\tilde N_j(\Jet^{m+1}f( a))$ are linearly independent, Lemma~\ref{lem6177fbb7} ensures that the map $F_G\circ\Jet^mf:G\to G$ is a diffeomorphism from a neighborhood $U_1$ of $ a$ to a neighborhood $U$ of $F_G\circ\Jet^mf( a)$ in $G$.
	Let $\phi:U\to U_1$ be its inverse.
	The  computations above, now read as
	\begin{equation}\label{eq06011620}
	\tilde v_j|_{U_1} = d\phi\circ\tilde N_j\circ \Jet^{m+1}f|_{U_1} .
	\end{equation}
	
	Define $h:U\to W$ by $h=F^0\circ\Jet^mf\circ \phi$.
	We shall prove that, for all $0\le k\le m$, 
	\begin{equation}\label{eq06011603}
	(\Jet^mh)^k = F^k\circ\Jet^mf\circ\phi ,
	\end{equation}
	which then implies~\eqref{eq06011605}.
	
	We prove~\eqref{eq06011603} by induction on $k$.
	For $k=0$, the identity~\eqref{eq06011603} is the definition of $h$. So we assume~\eqref{eq06011603} holds for $k=\ell \leq m-1$ and consider it for $k=\ell+1\le m$.
	Let $b\in U$, $a=\phi(b)$ and $j\in\{1,\dots,r\}$.
	Then
%	\begin{align*}
	{\small
	\begin{equation}\label{eq5eafdd4f}
	\begin{aligned}
	N_j(a,A^{\le m+1}_{f,a})\rcontr (\Jet^mh)^{\ell+1}(b) 
	&\overset{(i)}= \left.\frac{\dd}{\dd t}\right|_{t=0} (\Jet^mh)^{\ell}(b\exp(tN_j(a,A^{\le m+1}_{f,a})))\\
	&\overset{(ii)}= \left.\frac{\dd}{\dd t}\right|_{t=0} (F^\ell\circ\Jet^mf)(\phi(b\exp(tN_j(a,A^{\le m+1}_{f,a}))))\\
	&\overset{(iii)}= \left.\frac{\dd}{\dd t}\right|_{t=0} (F^\ell\circ\Jet^mf)(\phi(b)\exp(tv_j))\\
	&\overset{(iv)}= d F^\ell|_{(a,A^{\le m}_{f,a})}\left ((\tilde v_j(a),v_j\rcontr A^{\le m+1}_{f,a})\right ) \\
	&\overset{(v)}= N_j(a,A^{\le m+1}_{f,a})\rcontr F^{\ell+1}(a,A^{\le m}_{f,a}) ,
	\end{aligned}
	\end{equation}
	}
%	\end{align*}
	where: 
	$(i)$ follows from~\eqref{eq05161949},
	$(ii)$ is the inductive hypothesis,
	$(iii)$ follows from~\eqref{eq06011620},
	$(iv)$ is simply the computation of the differential
	and $(v)$ follows from~\eqref{eq06011622}.
	Since the vectors $N_j$ form a basis of $V_1$, we conclude that $(\Jet^mh)^{\ell+1}(b) = F^{\ell+1}(a,A^{\le m}_{f,a})$, that is,~\eqref{eq06011603} holds for $k=\ell+1$.
\end{proof}

\begin{proof}[Proof of Theorem~\ref{thm5eafe9c5}]
	The determining conditions at \eqref{prolongmtomplusone} become
	\begin{equation}\label{eq5eafde7c}
	d  \hat F^m( \hat p)   {\mathbb{\hat X}}_j(\hat p)
	= N_j(\hat p)\rcontr \hat F^{m+1}(\hat p) \quad \forall j ,
	\end{equation}
	where $N_j$ are defined in~\eqref{eq5eafda06}.
	Given that $\{ N_j(\hat p):j=1,\dots, r \}$ is a basis of $V_1$ 
	there exists a unique $\hat F^{m+1}(\hat p)\in \Lin^{m+1}(V_1;W)$ that satisfies~\eqref{eq5eafde7c}.
	
	We need to show that $\hat F^{m+1}(\hat p)\in \HorDer^{m+1}(\frk g;W)$.
	Write $\hat p = (a,A^{\le m+1})\in\hat \Omega$, let $f:G\to W$ be a smooth function with $A^{\le m+1}=A_{f,a}^{\le m+1}$, which exists by Corollary~\ref{cor5eafe914}, and let $h$ be a smooth function as in Lemma~\ref{lem5eafdaf1}. 
	We can perform again the computations in~\eqref{eq5eafdd4f} with $\ell=m$ until the second to last step and obtain
	\[
	N_j(a,A^{\le m+1}_{f,a})\rcontr (\Jet^{m+1}h)^{m+1}(F_G(a,A^{\le m}))
	= d F^m|_{(a,A^{\le m}_{f,a})}[(\tilde v_j(a),v_j\rcontr A^{\le m+1}_{f,a})] ,
	\]
	which is~\eqref{eq5eafde7c}
	with $\hat F^{m+1}(\hat p) = (\Jet^mh)^{m+1}(F_G(\hat p)) \in \HorDer^{m+1}(\frk g;W)$.
	
	Finally, suppose that $F$ is a diffeomorphism.
	Then $dF(\scr H^m_p)=\scr H^m_{F(p)}$ for all $p\in\Omega$ and thus $F^{-1}$ is also a contact map. Let $\hat\Omega_F$ denote the domain of $\hat F$ and let $\hat\Omega_{F^{-1}}$ denote the domain of the prolongation of $F^{-1}$ as defined above.
	
	We claim that $\hat F(\hat\Omega_F) \subset \hat\Omega_{F^{-1}}$.
	Indeed, let $\hat p\in\hat\Omega_F$ and set $\hat q=\hat F(\hat p)$.
	By the previous discussion, we know that $\hat q = \Jet^{m+1}h( a)$, where $h$ is given by Lemma~\ref{lem5eafdaf1}.
	Now notice that by~\eqref{eq619de09a}, the linear independence of the vectors $\tilde N_j(\hat p)$ is equivalent to the non-singularity of the differential $d(\pi_G\circ F\circ\Jet^mf)( a)$.
	Since the diagram~\eqref{eq619de3b7} commutes, the differential $d(\pi_G\circ F^{-1}\circ\Jet^mh)( b)$, where $ b=\pi_G(\hat q)$, is also non-singular.
	By~\eqref{eq619de09a} again, this time with $F^{-1}$ replacing $F$, we conclude that the vectors $\tilde N_j(\hat q)$ defined for $F^{-1}$ are also linearly independent and thus $\hat q\in \hat \Omega_{F^{-1}}$. 
	
	Let $H:\hat\Omega_{F^{-1}}\to\Jet^{m+1}(G;W)$ be the prolongation of $F^{-1}$.
	Since $H\circ\hat F$ is a prolongation of $F^{-1}\circ F = \Id_{\Omega}$, and since the prolongation of a map is unique, then $H\circ\hat F = \Id_{\hat\Omega_F}$, i.e., $H=\hat F^{-1}$ on $\hat F(\hat\Omega_F)\subset\hat\Omega_{F^{-1}}$ and thus $\hat F$ is a diffeomorphism.
	By symmetry, we get also 
	$H(\hat\Omega_{F^{-1}})\subset\hat\Omega_F$ and $H^{-1}=\hat F$.
	In particular, $\hat F(\hat\Omega_F)=\hat\Omega_{F^{-1}}$.
\end{proof}

%%%%%%%%%%%%%%%%%%%%%%%%%%%%%%%%%%%%%%%%%%%%%%%%%%%%%%%%%%%%%%%
%%%%%%%%%%%%%%%%%%%%%%%%%%%%%%%%%%%%%%%%%%%%%%%%%%%%%%%%%%%%%%%
\section{De-prolongation of Contact Maps: B\"acklund Theorem}\label{sec616fa813}

\subsection{The de-prolongation theorems}
This section is devoted to the de-pro\-lon\-ga\-tion Theorem~\ref{thm5ebc31f5}, which extends the well known B\"acklund Theorem to the jet spaces discussed here.

Let $G$ be a stratified Lie group with stratified Lie algebra $\frk g=\bigoplus_{j=1}^s V_j$, and $W$ a vector space.
The first part of Theorem~\ref{thm5ebc31f5} follows by iterating the following result.

\begin{theorem}\label{thm5ebc3248}
	Assume $m\ge1$.
	Denote by $\pi_m:\Jet^{m+1}(G;W)\to \Jet^{m}(G;W)$ the projection along $\HorDer^{m+1}(\frk g;W)$.
	Let $\Omega\subset\Jet^{m+1}(G;W)$ be an open set, define $\Omega'=\pi_m(\Omega)$ and assume that $\pi_m^{-1}(q)\cap\Omega$ is connected for every $q\in\Omega'$.
	
	If $F:\Omega\to\Jet^{m+1}(G;W)$ is a contact diffeomorphism,
	then there is a contact diffeomorphism $F':\Omega'\to\Jet^{m}(G;W)$ such that $\pi_m\circ F=F'\circ\pi_m$, i.e., the following diagram commutes:
	\[
	\xymatrix{
	\Omega \ar[d]_{\pi_m}\ar[r]^{F} & \Jet^{m+1}(G;W) \ar[d]^{\pi_m} \\
	\Omega'\ar[r]_{F'} & \Jet^{m}(G;W)
	}
	\]
\end{theorem}

The following theorem states that, with further assumptions, contact diffeomorphisms of $\Jet^1(G;W)$ are prolongations of contact diffeomorphisms of $\Jet^0(G;W)$.
This is a more precise restatement of the second part of Theorem~\ref{thm5ebc31f5}.

\begin{theorem}\label{thm606f2a10}
	Suppose that one of the following conditions is satisfied:
	\begin{enumerate}[label=(\Alph*)]
	\item\label{item5ebc765a}
	$\dim(W)>1$, or
	\item\label{item5ebc765b}
	$\dim(W)=1$ and for every $v\in V_1\setminus\{0\}$ there is $v'\in V_1$ with $[v,v']\neq0$.
	\end{enumerate}
	
	Denote by $\pi_0:\Jet^1(G;W)\to \Jet^0(G;W)$ the projection along $\HorDer^1(\frk g;W)$.
	Let $\Omega\subset\Jet^1(G;W)$ be an open set, define $\Omega'=\pi_0(\Omega)$ and assume that $\pi_0^{-1}(q)\cap\Omega$ is connected for every $q\in\Omega'$.
	
	If $F:\Omega\to\Jet^1(G;W)$ is a contact diffeomorphism,
	then there is a contact diffeomorphism $F':\Omega'\to\Jet^0(G;W)$ such that $\pi_0\circ F=F'\circ\pi_0$, i.e., the following diagram commutes:
	\begin{equation}\label{eq616fa757}
	\begin{gathered}
	\xymatrix{
	\Omega \ar[d]_{\pi_0}\ar[r]^{F} & \Jet^1(G;W) \ar[d]^{\pi_0} \\
	\Omega'\ar[r]_{F'} & \Jet^0(G;W)
	}
	\end{gathered}
	\end{equation}
\end{theorem}

\begin{remark}\label{rem616fabbb}
	Theorem~\ref{thm606f2a10} is sharp, that is, {\it if both conditions~\ref{item5ebc765a} and~\ref{item5ebc765b} are violated, then there exists a contact diffeomorphism $F:\Jet^1(G;W)\to\Jet^1(G;W)$ that is not the lift of a contact diffeomorphism $\Jet^0(G;W)\to\Jet^0(G;W)$.}
	
	The construction of such a contact map $F$ is the following.
	If both conditions~\ref{item5ebc765a} and~\ref{item5ebc765b} are violated, then $W=\R$ and there is $\hat v\in V_1$ such that $[\hat v,V_1]=\{0\}$.
	It follows that $\frk g$ is the direct product of a stratified Lie algebra $\frk g'$ and $\R$ (take $V_1'$ such that $V_1=V_1'\oplus\R\hat v$ and $\frk g'$ the Lie span of $V_1'$).
	Then
	\[
	\jet^1(\frk g;W) \simeq (\frk g'\times\R) \oplus \R \oplus ((\frk g')^*\times\R^*) ,
	\]
	with Lie bracket
	\[
	[ ((v,x),z,(\alpha,y)) , ((\bar v,\bar x),\bar z,(\bar \alpha,\bar y)) ]
	 = (([v,\bar v],0) , \alpha(\bar v) - \bar \alpha(v) + (\bar y x - y\bar x) , (0,0)) .
	\]
	Define
	\[
	\phi((v,x),z,(\alpha,y)) = ((v,-y),z,(\alpha,x)) .
	\]
	Then $\phi$ is a Lie algebra automorphism of $\jet^1(\frk g;W)$. Indeed,
	\begin{align*}
	[ \phi((v,x),z,(\alpha,y)) , &\phi((\bar v,\bar x),\bar z,(\bar \alpha,\bar y)) ] \\
	&= [ ((v,-y),z,(\alpha,x)) , ((\bar v,-\bar y),\bar z,(\bar \alpha,\bar x)) ] \\
	&= (([v,\bar v],0) , \alpha(\bar v) - \bar \alpha(v) + (- y\bar x + \bar y x) , (0,0)) \\
	&= [ ((v,x),z,(\alpha,y)) , ((\bar v,\bar x),\bar z,(\bar \alpha,\bar y)) ] .
	\end{align*}
	
	The map $F:\Jet^1(G;W)\to\Jet^1(G;W)$ is the unique Lie group automorphism with $F_*=\phi$.
	Since $\phi(\jet^1(\frk g;W)_1)=\jet^1(\frk g;W)_1$, $F$ is a contact diffeomorphism.
	By construction, there is no $F'$ that makes the diagram~\eqref{eq616fa757} commute, not even locally.
\end{remark}

%%%%%%%%%%%%%%%%%%%%%%%%%%%%%%%%%%%%%%%%%%%%%%%%%%%%%%%%%%%%%%%
\subsection{Characteristic vector fields and proof of Theorem~\ref{thm5ebc3248}}
\label{sec616faee7}

De-prolongation results are based on the existence of a particular type of characteristic vector field.

For an open set $\Omega \subset \Jet^m(G;W)$, we let $T\Jet^m_\Omega=\cup_{p \in \Omega} T_p\Jet^m(G;W)$ and $\scr H^m_\Omega = \cup_{p \in \Omega} \scr H^m_p$. Starting with $L^1_\Omega=\Gamma(\scr H^m_\Omega)$, we define a filtration of $\Gamma(T\Jet^m_\Omega)$ inductively buy setting $L^{i+1}_\Omega=L^i_\Omega+[L^1_\Omega,L^i_\Omega]$ for $i=1,\dots,s-1$. If $F:\Omega \to \Jet^m(G;W)$ is a smooth contact diffeomorphism, then $F_*L^i_\Omega=L^i_{f(\Omega)}$ since $F_*L^1_\Omega=L^1_{F(\Omega)}$. A Cauchy characteristic of order $i$ over $\Omega$ , is a vector field $X \in \Gamma(L^i_\Omega)$ such that $[X, \Gamma(L^i_\Omega)] \subset  \Gamma(L^i_\Omega)$. The set $\mathcal{C}^i_\Omega$ of all Cauchy characteristics of order $i$ over $\Omega$ is a Lie algebra thanks to the Jacobi identity. If $F:\Omega \to \Jet^m(G;W)$ is a smooth contact diffeomorphism then $F_*\mathcal{C}^i(\Omega) \subseteq \mathcal{C}^i(F(\Omega))$ since 
\begin{align*}
F_*[X,\Gamma(L^i_\Omega)]&=[F_*X,F_*\Gamma(L^i_\Omega)]\\
&=[F_*X,\Gamma(L^i_{F(\Omega)})] \subset  F_*\Gamma(L^i_{\Omega})  = \Gamma(L^i_{F(\Omega)}).
\end{align*}

The particular notion of characteristic we use here, is a vector field $\tilde X \in \Gamma(L^1_\Omega)$ such $ [\tilde X,[\tilde X, \tilde Y]] \subset  \Gamma(L^2_\Omega)$ for all $\tilde Y \in \Gamma(L^1_\Omega)$. 
The characteristic property of these vector fields is preserved by contact transformation. 
Indeed, if $F:\Omega \to \Jet^m(G;W)$ is a smooth contact diffeomorphism, then
\begin{align*}
	[F_*\tilde X,[F_*\tilde X, \Gamma(L^1_{F(\Omega)})]]
	&=[F_*\tilde X,[F_*\tilde X, F_*\Gamma(L^1_\Omega)]]\\
	&= F_*[\tilde X,[\tilde X, \Gamma(L^1_\Omega)]] 
	\subset  F_*\Gamma(L^2_{\Omega})  
	= \Gamma(L^2_{F(\Omega)}).
\end{align*}

This observation together with the characterisation of characteristic vector fields proved in Lemma~\ref{lem5ebc2bc6} below, allow us to prove the de-prolongation result in Theorem~\ref{thm5ebc3248}. 

Once again, we will heavily use the conventions described in Section~\ref{sec616e79a0} to represent vector fields on Lie groups.
In particular, notice the difference between ``$[X_p,Y_p]$'' and ``$[X,Y]_p$''.

\begin{lemma}\label{lem5ebc2bc6}
	Assume $m\ge2$.
	Denote by $\Pi_3$ the projection $\jet^m(\frk g;W)\to \jet^m(\frk g;W)_3$ given by the stratification of $\jet^m(\frk g;W)$.
	Let $X:\Jet^m(G;W)\to\jet^m(\frk g;W)_1$ be a horizontal vector field and write $X=v^X+A^X$, where $v^X:\Jet^m(G;W)\to V_1$ and $A^X:\Jet^m(G;W)\to\HorDer^m(\frk g;W)$.
	
	Then the following are equivalent:
	\begin{enumerate}[label=(\roman*)]
	\item\label{item5ebc298a}
	$v^X=0$;
	\item\label{item5ebc298b}
	$X$ is characteristic, i.e., $\Pi_3([X,[X,Y]])=0$ for every horizontal vector field $Y:\Jet^m(G;W)\to\jet^m(\frk g;W)_1$.
	\end{enumerate}
\end{lemma}
\begin{proof}
	$\ref{item5ebc298a}\THEN\ref{item5ebc298b}$
	If $Y:\Jet(G;W)\to\jet^m(G;W)_1$ is a horizontal vector field and we write $Y=v^Y+A^Y$,
	then for every $p\in \Jet(G;W)$,
	\begin{equation}\label{eqcoso}
	\begin{aligned}
	\Pi_3([X,[X,Y]]_p)
	&\overset{(*)}= \Pi_3([X_p,[X_p,Y_p]]) \\
	&= [A^X_p,[A^X_p,v^Y_p]] + [A^X_p,[A^X_p,A^Y_p]]
	= 0 ,
	\end{aligned}
	\end{equation}
	 where the identity $(*)$ is justified by writing $X$ and $Y$ in a basis of $\jet^m(G;W)_1$ and then applying~\eqref{eq616e7ae8}.
	
	$\ref{item5ebc298b}\THEN\ref{item5ebc298a}$
	Let $B\in\HorDer^m(\frk g;W)$ and let $Y=A^Y \equiv B$ be a constant vector field.
	Then, for every $p\in G$,
	\[
	0
	= \Pi_3([X,[X,Y]]_p) 
	= \Pi_3([X_p,[X_p,Y_p]]) 
	= [v^X,[v^X,B]] 
	= v^X\rcontr(v^X\rcontr B).
	\]
	Since $B$ is arbitrary in $\HorDer^m(\frk g;W)$ and $m\ge2$, we conclude that $v^X=0$.
\end{proof}

\begin{proof}[Proof of Theorem~\ref{thm5ebc3248}]
	Let $\scr V\subset T\Jet^{m+1}(G;W)$ be the left-invariant vector bundle defined by $\HorDer^{m+1}(\frk g;W)$.
	
	We claim that $dF_p(\scr V)=\scr V_{F(p)}$ for every $p\in \Omega$.
	Let $A_1,\dots,A_\ell$ be a basis of $\HorDer^{m+1}(\frk g;W)$ and consider the corresponding left-invariant vector fields $\tilde A_1,\dots,\tilde A_\ell$, which form a frame for $\scr V$.
	Fix $k\in\{1,\dots,\ell\}$.
	By Lemma~\ref{lem5ebc2bc6}, $\tilde A_k$ is a characteristic vector field.
	Since $F$ is contact, $F_*\tilde A_k$ is also a characteristic vector field.
	By Lemma~\ref{lem5ebc2bc6} that $F_*\tilde A_k|_{F(p)}\in\scr V_{F(p)}$ for every $p\in\Omega$.
	We have that $d F_p (\tilde A_1|_p),\dots,d F_p(\tilde A_\ell|_p)$ are linearly independent and belong to $\scr V_{F(p)}$ and thus they form a basis of this vector space.
	The claim is thus proven.

	Finally, notice that the fibers $\pi_m^{-1}(q)$, for $q\in\Jet^m(G;W)$, are the integral manifolds of $\scr V$.
	We conclude that there exists a smooth map $F'$ so that the above diagram commutes.
	The fact that $F'$ is contact map follows from Lemma~\ref{lem5ebc4d52}.
\end{proof}

%%%%%%%%%%%%%%%%%%%%%%%%%%%%%%%%%%%%%%%%%%%%%%%%%%%%%%%%%%%%%%%
\subsection{Proof of Theorem~\ref{thm606f2a10} with condition~\ref{item5ebc765a}}

\begin{lemma}\label{lem5ebc8077}
	Let $X,Y:\Jet^m(G;W)\to\jet^m(G;W)_1$ be horizontal vector fields and $p\in\Jet^m(G;W)$.
	The commutator $[X,Y]_p$ is  horizontal if and only if $[X_p,Y_p]=0$.
\end{lemma}
\begin{proof}
	Write $X$ and $Y$ in a basis of $\jet^m(G;W)_1$ and then apply \eqref{eq616e7ae8}.
\end{proof}

\begin{corollary}\label{cor5ebc8112}
	If $\scr V\subset \scr H^m$ is a horizontal involutive subbundle, then, for every $p\in \Jet^m(G;W)$, the space $d L_p^{-1}(\scr V_p)$ is an abelian subalgebra of $\jet^m(G;W)_1$.
\end{corollary}

\begin{proposition}\label{prop5ebc819c}
	Assume condition~\ref{item5ebc765a}, that is, $\dim(W)>1$.
	If $R$ is an abelian subalgebra of $\jet^1(G;W)$ contained in $\jet^1(G;W)_1=V_1\times\HorDer^1(\frk g;W)$ such that $ \dim(R)=\dim(V_1)\cdot\dim(W)$, then $R=\HorDer^1(\frk g;W)$.
\end{proposition}
\begin{proof}
	Let $m=\dim(V_1)$ and $n=\dim(W)>1$, so that $\dim(R)=mn$.
	Notice that $\HorDer^1(\frk g;W)=\Lin(V_1;W)$ and thus $\dim(\HorDer^1(\frk g;W))=mn$.
	
	Let $\pi:V_1\times\HorDer^1(\frk g;W)\to V_1$ be the projection to the first factor and set $a=\dim(\pi(R))$ and $b = \dim(R\cap\HorDer^1(\frk g;W)) $.
	Notice that $\dim(R)=a+b$ since the restriction of $\pi$ to $R$ is a linear map with kernel of dimension $b$ and image of dimension $a$.
	Moreover,
	\begin{align*}
		b 
		= \dim(R\cap\HorDer^1(\frk g;W)) 
		&= mn + mn - \dim(R+\HorDer^1(\frk g;W)) \\
		&\ge mn + mn - (mn+m) = mn-m .
	\end{align*}
	
	Let $v_1,\dots,v_a\in V_1$ be a basis of $\pi(R)$ and define the map $\phi:\HorDer^1(\frk g;W)\to W^a$ as $\phi(\alpha)=(\alpha(v_1),\dots,\alpha(v_a))$. 
	Then $\phi$ is surjective and thus $\dim(\ker\phi)=mn-an$.
	
	Since $R$ is an abelian Lie algebra, 
	if $(0,\beta)\in R\cap\HorDer^1(\frk g;W)$ and $(v_j,\alpha)\in R$ for some $j$, then
	\[
	0 = [(0,\beta),(v_j,\alpha)] = \beta(v_j).
	\]
	Therefore, $\phi(\beta) = 0$, and so 
	$b\le \dim(\ker\phi) = mn-an$.
	All in all, we obtain
	\[
	mn = \dim(R) = a+b \le a + mn - an ,
	\quad\text{that is,}\quad
	0 \le a(1-n) .
	\]
	Since $n>1$, then $a$ must be zero.	
\end{proof}

Theorem~\ref{thm606f2a10} with condition~\ref{item5ebc765a} is a special case of the following proposition.

\begin{proof}[Proof of Theorem~\ref{thm606f2a10} with condition~\ref{item5ebc765a}]
	Assume condition~\ref{item5ebc765a}, i.e., $\dim(W)>1$.
	Let $\scr V\subset T\Jet^1(G;W)$ be the left-invariant vector bundle defined by $\HorDer^1(\frk g;W)$.
	We claim that $dF_p(\scr V)=\scr V_{F(p)}$ for every $p\in \Omega$.
	Indeed, $\scr V$ is a horizontal involutive  subbundle of rank $\dim(V_1)\cdot\dim(W)$, and so $d F(\scr V)$ is also a horizontal involutive  subbundle of the same rank since $F$ is a contact diffeomorphism.
	From Corollary~\ref{cor5ebc8112}, we obtain that for every $q$ in the image of $F$, $dF_q(\scr V)$ is the left translation of an abelian subalgebra $R_q$ in $\jet(\frk g;W)_1$.
	By Proposition~\ref{prop5ebc819c}, we get $R_q=\HorDer^1(\frk g;W)$, i.e., $d F_q(\scr V)=\scr V_q$, as claimed.
	
	From this point, the proof concludes as for Theorem~\ref{thm5ebc3248} in Section~\ref{sec616faee7}.
\end{proof}

%%%%%%%%%%%%%%%%%%%%%%%%%%%%%%%%%%%%%%%%%%%%%%%%%%%%%%%%%%%%%%%
\subsection{Proof of Theorem~\ref{thm606f2a10} with condition~\ref{item5ebc765b}}

\begin{proof}[Proof of Theorem~\ref{thm606f2a10} with condition~\ref{item5ebc765b}]
	Assume condition~\ref{item5ebc765b}, i.e., $\dim(W)=1$ and for every $v\in V_1\setminus\{0\}$ there is $v'\in V_1$ with $[v,v']\neq0$.
	It follows that $\HorDer(\frk g;W)=V_1^*$ and $\jet^1(\frk g;W)_1=V_1\oplus V_1^*$, while $\jet^1(\frk g;W)_2=V_2\oplus W$.
	For every $X\in\jet(\frk g;W)_1$, define 
	\[
	R_X=\{Y\in \jet(\frk g;W)_1:[X,Y]=0\}
	\text{ and }
	\delta(X) = \dim R_X .
	\]
	
	If $X=(0, \alpha)$, where  $\alpha\in V_1^*$ is nonzero and $Y=(w,\beta)\in\jet(\frk g;W)_1$, then
	$[X,Y] = \alpha(w)$ is zero if and only if $w\in\ker(\alpha)$.
	Therefore, 
	\begin{equation}\label{eq61dda759}
	\delta(0,\alpha)=2\dim(V_1)-1,
	\qquad\text{ whenever $\alpha\neq0$.}
	\end{equation}
%	$\delta(0,\alpha)=2\dim(V_1)-1$, whenever $\alpha\neq0$.
	
	If $X=(v,\alpha)$ with $v\neq0$ and $Y=(w,\beta)$, then
	$[X,Y]=0$ if and only if $[v,w]=0$ and $\alpha(w)-\beta(v)=0$.
	Since ${\rm ad}_v$ is nontrivial on $V_1$ by assumption, the projection of $R_X$ to $V_1$ has dimension at most $\dim(V_1)-1$.
	Moreover, each fiber in $R_X$ of this projection has dimension $\dim(V_1)-1$, because $\alpha(w)-\beta(v)=0$ is a nontrivial linear equation in $\beta$, when $w$ is fixed.
	We conclude that 
	\begin{equation}\label{eq61dda774}
	\delta(v,\alpha) \le 2\dim(V_1)-2,
	\qquad\text{ whenever $v\neq0$.}
	\end{equation}
%	$\delta(v,\alpha) \le 2\dim(V_1)-2$, whenever $v\neq0$.
	
	Fix $p\in\Omega$.
	We claim that, for every $X\in\jet(\frk g;W)_1$,
	\begin{equation}\label{eq5ebc89a5}
	\delta(F_*\tilde X|_{F(p)}) = \delta(X) .
	\end{equation}
	
	If $Y\in \jet(\frk g;W)_1$ is such that $[X,Y]=0$,
	then $F_*([\tilde X,\tilde Y]) = [F_*\tilde X,F_*\tilde Y] = 0$ because $F$ is a diffeomorphism.
	Since $F$ is contact, then both $F_*\tilde X$ and $F_*\tilde Y$ are horizontal vector fields.
	Thus, we get from Lemma~\ref{lem5ebc8077} that $ [(F_*\tilde X)_{F(p)},(F_*\tilde Y)_{F(p)}] = 0$.
	Therefore, $F_*\tilde Y_{F(p)} \in R_{F_*\tilde X_{F(p)}}$.
	Since $F^{-1}$ is also a contact diffeomorphism we conclude \eqref{eq5ebc89a5}.
	
	From~\eqref{eq5ebc89a5} and the previous computation of $\delta(X)$ in~\eqref{eq61dda759} and~\eqref{eq61dda774}, we obtain that if $\alpha\in V_1^*$ then $F_*\tilde\alpha_{F(p)}\in V_1^*$ for every $p\in\Omega$.
	
	From this point, the proof concludes as for Theorem~\ref{thm5ebc3248} in Section~\ref{sec616faee7}.
\end{proof}

%%%%%%%%%%%%%%%%%%%%%%%%%%%%%%%%%%%%%%%%%%%%%%%%%%%%%%%%%%%%%%%
%%%%%%%%%%%%%%%%%%%%%%%%%%%%%%%%%%%%%%%%%%%%%%%%%%%%%%%%%%%%%%%
\section{Embedding of Carnot Groups into Jet Spaces}\label{sec61719991}
In this section we prove Theorem~\ref{thm6171c00e} and we compute its application to groups of step 2 and 3.

\subsection{Proof of Theorem~\ref{thm6171c00e}}
	For this proof, we will extensively use the identification of the group with its Lie algebra via the exponential map.
	More explicitly, if $\frk g$ is the Lie algebra of (a nilpotent, simply connected Lie group) $G$, then we define for $x,y\in\frk g$
	\[
	xy = \log(\exp(x)\exp(y)) ,
	\]
	which, via the BCH formula, has an explicit polynomial expression.
	If $f$ is a smooth function and $v\in\frk g$, then the corresponding left and right invariant vector fields take the form 
	\[
	\tilde vf(p) = \left.\frac{\dd}{\dd t}\right|_{t=0} f(p(tv))
	\quad\text{ and }\quad
	 v^\dagger f(p) = \left.\frac{\dd}{\dd t}\right|_{t=0} f((tv)p).
	\]
	
	Let $G$ be the stratified group from Theorem~\ref{thm6171c00e} with stratified Lie algebra $\frk g=\oplus_{j=1}^{s+1}V_j$.
	For $x\in\frk g$, we write $x_j=\Pi_j(x)\in V_j$ and $x'=\sum_{j\le s}x_j\in\frk g'$.
	The BCH formula gives a map $\eta:\frk g\times\frk g\to V_{s+1}$, which is a polynomial, such that
	\[
	\Pi_{s+1} (xy) = x_{s+1} + y_{s+1} + \eta(x,y) .
	\]
	The function $\eta$ has the following properties:
	\begin{enumerate}
	\item
	$\eta(x,y)$ depends only on $x'$ and $y'$;
	\item
	$\eta(\delta_\lambda x,\delta_\lambda y) = \lambda^{s+1} \eta(x,y)$;
	\item
	since $$\Pi_{s+1}((xy)z) = x_{s+1}+y_{s+1}+z_{s+1}+\eta(x,y)+\eta(xy,z)$$ and
	$$\Pi_{s+1}(x(yz)) = x_{s+1}+y_{s+1}+z_{s+1}+\eta(y,z)+\eta(x,yz),$$ we have
	\begin{equation}\label{eq6171bbce}
	\eta(xy,z) + \eta(x,y) = \eta(x,yz) + \eta(y,z).
	\end{equation}
	\end{enumerate}
	
	Define $\frk g'  = \oplus_{j=1}^{s}V_j$ with Lie brackets $[x,y]'=[x,y]-\Pi_{s+1}([x,y])$, i.e.,  $\frk g'\simeq \frk g/V_{s+1}$.
	We identify again the group $G'$ with $\frk g'$ endowed with the group operation given by the BCH formula.
	One can easily check that 
	\[
	(xy)' = (x'y')' = (x'y')_{\frk g'} ,
	\]
	where the last term is the group operation given by the BCH formula on $\frk g'$.
	
	By Theorem~\ref{thm6171c9a8}, the polynomial jet space $\Jet_\Poly^s(G';W)$ is equivalent to the jet space $\Jet^s(G';W)$ defined in Section~\ref{sec5ebc3ffe}.
	We will define an injective morphism of stratified Lie algebras $\phi:\frk g\to\jet_\Poly^s(G';V_{s+1})$.
	We consider linear maps $\phi$ of the following form: 
	For every $k\in\{1,\dots,s+1\}$ there is a linear map $\phi_k:V_k\to \Poly_{e_{G'}}^{s+1-k}(G';V_{s+1})$ such that for all $v\in V_k$
	\begin{equation}\label{eq6171be12}
	\phi(v) = v + \phi_k(v) \in V_k\oplus \Poly_{e_{G'}}^{s+1-k}(G';V_{s+1})
	= \jet_\Poly^s(\frk g';V_{s+1})_k .
	\end{equation}
	Notice that $\phi_{s+1}:V_{s+1}\to V_{s+1}$.
	 Any such map $\phi$ is a morphism of stratified Lie algebras 
	if and only if it is a Lie algebra morphism, because it already preserves the stratification.
	Furthermore, $\phi$ is a Lie algebra morphism
	if and only if for every $v\in V_i$ and $w\in V_j$, we have $[\phi(v),\phi(w)] = \phi([v,w])$, that is, by~\eqref{eq61781ab0},
	\begin{enumerate}
	\item
	if $i+j<s+1$,
	\begin{equation}\label{eq61716e63_a}
	w\rcontr\phi_i(v) - v\rcontr\phi_j(w) = \phi_{i+j}([v,w]) .
	\end{equation}
	\item
	if $i+j=s+1$, 
	\begin{equation}\label{eq61716e63_b}
	w\rcontr\phi_i(v) - v\rcontr\phi_j(w) = [v,w] + \phi_{s+1}([v,w]) .
	\end{equation}
	\end{enumerate}	
	
	 For $v\in V_k$, $1\le k\le s+1$, we define
	\begin{equation}\label{eq6171be3c}
	\phi_k(v) = ( v^\dagger_x \eta)(e,y) ,
	\end{equation}
	by which we mean that $\phi_k(v)$ is the polynomial in $y$ resulting from deriving $\eta$ in $x$ along the right invariant vector field $v^\dagger$ and evaluating at $x=e$.
	Similarly, we will also write $v^\dagger_y$ to denote the derivation in $y$.
	Notice that taking the right-invariant vector field $v^\dagger$ is pretentious, because
	\[
	( v^\dagger_x \eta)(e,y) 
	= \left.\frac{\dd}{\dd t}\right|_{t=0} \eta(tv,y) 
	= ( \tilde v_x \eta)(e,y) .
	\]
	Notice also that, since $\eta$ depends only on $y'$, $\phi_k(v)$ is actually a function on $G'$. 
	
	We claim that $( v^\dagger_x \eta)(e,y) \in \Poly^{s+1-k}_{e_{G'}}(G';V_{s+1})$ when $v\in V_k$.
	Indeed,
	\begin{align*}
	(v^\dagger_x \eta)(e,\delta_\lambda y) 
	&= \left. \frac{\dd}{\dd t}\right|_{t=0} \eta((tv),\delta_\lambda y) \\
	&= \left. \frac{\dd}{\dd t}\right|_{t=0} \eta(\delta_\lambda (\lambda^{-k}tv),\delta_\lambda y) \\
	&= \lambda^{s+1} \left. \frac{\dd}{\dd t}\right|_{t=0} \eta((\lambda^{-k}tv), y)\\
	&= \lambda^{s+1-k} \left. \frac{\dd}{\dd t}\right|_{t=0} \eta((tv), y)
	= \lambda^{s+1-k} (v^\dagger_x \eta)(e, y) .
	\end{align*}
	Therefore, $( v^\dagger_x \eta)(e,y) \in \Poly^{s+1-k}_{e_{G'}}(G';V_{s+1})$, and $\phi$ is therefore well defined.
	
	Next, we show that $\phi$ does in fact satisfy~\eqref{eq61716e63_a} and~\eqref{eq61716e63_b}.
	Let $v\in V_i$ and $w\in V_j$ for $1\le i,j\le s+1$.
	Then, using~\eqref{eq6171bbce},
	\begin{align*}
	w\rcontr\phi_i(v) (y)
	&= w^\dagger_y(v^\dagger_x \eta)(e,y) \\
	&= \left. \frac{\dd}{\dd t}\right|_{t=0} (v^\dagger_x \eta)(e,(tw)y) \\
	&= \left. \frac{\dd}{\dd t}\right|_{t=0}\left. \frac{\dd}{\dd r}\right|_{r=0} 
		\eta(rv,(tw)y) \\
	&= \left. \frac{\dd}{\dd t}\right|_{t=0}\left. \frac{\dd}{\dd r}\right|_{r=0}
		\left( \eta((rv)(tw),y) + (rv,tw) - \eta(tw,y) \right) \\
	&= (w^\dagger_x v^\dagger_x \eta)(e,y) 
		+ \left. \frac{\dd}{\dd t}\right|_{t=0}\left. \frac{\dd}{\dd r}\right|_{r=0}\eta(rv,tw) .
	\end{align*}
	Notice that 
	\[
	\eta(rv,tw) - \eta(tw,rv)
	= \Pi_{s+1}\left( \log(\exp(rv)\exp(tw)) - \log(\exp(tw)\exp(rv)) \right) 
	\]
	and that the BCH formula implies that
	\[
	\log(\exp(rv)\exp(tw)) = rv + tw + \frac{rt}{2} [v,w] + P(rv,tw)
	\]
	where $\left. \frac{\dd}{\dd t}\right|_{t=0}\left. \frac{\dd}{\dd r}\right|_{r=0} P(rv,tw) = 0$.
	Therefore
	\[
	\left. \frac{\dd}{\dd t}\right|_{t=0}\left. \frac{\dd}{\dd r}\right|_{r=0}
	\left( \eta(rv,tw) - \eta(tw,rv) \right)
	= \Pi_{s+1}([v,w]) .
	\]
	We conclude that
	\begin{align*}
	(w\rcontr\phi_i(v) - v\rcontr\phi_j(w))(y)
	&= (w^\dagger_x v^\dagger_x \eta)(e,y) - (v^\dagger_x w^\dagger_x \eta)(e,y) + \Pi_{s+1}([v,w]) \\
	&= (-[v^\dagger,w^\dagger]_x\eta)(e,y) + \Pi_{s+1}([v,w]) \\
	&= ([v,w]^{\dagger}_x\eta)(e,y) + \Pi_{s+1}([v,w]) ,
	\end{align*} 
	where we used the standard relation 
	$-[v^\dagger,w^\dagger] = [v,w]^{\dagger}$
	between right and left invariant vector fields.
	
	We have thus shown that the map $\phi$ defined as in~\eqref{eq6171be12} with~\eqref{eq6171be3c} is a morphism of stratified Lie algebras $\phi:\frk g\to\jet^s(\frk g';V_{s+1})$.
	The only missing property we need to conclude is that $\phi$ is injective.
	Since $\phi|_{V_k}$ is injective for every $k<s+1$ and since $\phi$ preserves the stratification, we only need to check that $\phi|_{V_{s+1}}$ is injective.
	But $\phi|_{V_{s+1}}$ is the identity, because if $v\in V_{s+1}$, then $(v^\dagger_x\eta)(e,y) = 0$, as $\eta$ does not depend on $x_{s+1}$.
\qed

%%%%%%%%%%%%%%%%%%%%%%%%%%%%%%%%%%%%%%%%%%%%%%%%%%%%%%%%%%%%%%%
\subsection{Example: step 2}
As an example, we will show how stratified groups of step 2 embed into the standard jet spaces over abelian groups.
Let $\frk g = V_1\oplus V_2$ be a stratified Lie algebra of step 2.
Then
\[
\jet^1(V_1;V_2) = V_1 \oplus V_2 \oplus \Lin(V_1;V_2)
\]
with Lie brackets
\[
[ (v_1,v_2,A) , (w_1,w_2,B) ]
= (0, A(w_1) - B(v_1) , 0 ) .
\]
Notice that
\[
\log(\exp(v)\exp(w)) = v+w+\frac12[v,w]
= (v_1+w_1) + (v_2+w_2 + \frac12[v_1,w_1] ) ,
\]
so that $\eta(v,w) = \frac12[v_1,w_1]$.
The embedding constructed in Theorem~\ref{thm6171c00e} is
\[
\phi(v_1,v_2) = (v_1,v_2,\frac12[v_1,\cdot]) .
\]

%%%%%%%%%%%%%%%%%%%%%%%%%%%%%%%%%%%%%%%%%%%%%%%%%%%%%%%%%%%%%%%
\subsection{Example: step 3}
Let $\frk g = V_1\oplus V_2\oplus V_3$ be a stratified Lie algebra of step~3.
Then $\frk g' = V_1 \oplus V_2$ and
\[
\jet^2(V_1\oplus V_2;V_3) = (V_1 \oplus V_2) \oplus V_3 \oplus \Lin(\frk g';V_3) \oplus \HorDer^2(\frk g';V_3)
\]
with Lie brackets
\begin{multline*}
[ (v_1,v_2;v_3;A^1;A^2) , (w_1,w_2;w_3;B^1;B^2) ] \\
= (0,[v_1,w_1]; A^1(w_1) - B^1(v_1) + A^2(w_2) - B^2(v_2) ; A^2(w_1) - B^2(v_1) ; 0 ) .
\end{multline*}
Notice that
\begin{multline*}
\log(\exp(v)\exp(w)) = v+w+\frac12[v,w] + \frac1{12} ([v,[v,w]] + [w,[w,v]]) \\
= (v_1+w_1) + (v_2+w_2 + \frac12[v_1,w_1] ) +\hfill \\
 + (v_3+w_3+\frac12([v_1,w_2]+[v_2,w_1]) + \frac1{12}([v_1,[v_1,w_1]] + [w_1,[w_1,v_1]])
\end{multline*}
so that 
\[
\eta(v,w) = \frac12([v_1,w_2]+[v_2,w_1]) + \frac1{12}([v_1,[v_1,w_1]] + [w_1,[w_1,v_1]] .
\]
Therefore,
\begin{align*}
\phi_1(v_1)(w)
	&= \left. \frac{\dd}{\dd t}\right|_{t=0} \eta(tv_1,w)
	= \frac12 [v_1,w_2] + \frac1{12} [w_1,[w_1,v_1] , \\
\phi_2(v_2)(w)
	&= \left. \frac{\dd}{\dd t}\right|_{t=0} \eta(tv_2,w)
	= \frac12 [v_2,w_1] , \\
\phi_3(v_3)(w)
	&= \left. \frac{\dd}{\dd t}\right|_{t=0} \eta(tv_3,w) 
	= 0 \text{ (as expected).}
\end{align*}
We obtain
\[
\phi(v_1,v_2,v_3) =
( v_1,v_2;v_3;\frac12 [v_2,\Pi_1(\cdot)]; \frac12 [v_1,\Pi_2(\cdot)] + \frac1{12} [\Pi_1(\cdot),[\Pi_1(\cdot),v_1] ) ,
\]
as a map into $\Jet_{\Poly}^2(\frk g';V_3)$.

If we want to write the components of $\phi$ as multilinear maps (instead of polynomials),
we can use Theorem~\ref{thm6171c9a8}.
So, if $P:G'\to V_3$ is the map $P(x) = [\Pi_1(x),[\Pi_1(x),v_1]]$, then we need to compute
$A^2_{P,e}(x_1,y_1) = \tilde y_1\tilde x_1P(e)$, which we can compute using~\eqref{eq05162009}.
We obtain
\[
A^2_{P,e}(x_1,y_1) = [x_1,[y_1,v_1]] + [y_1,[x_1,v_1]]
\]
Similarly, for $Q(x) = [v_1,\Pi_2(x)]$, we get
$A^2_{Q,e}(x_1,y_1) = \frac12 [v_1,[y_1,x_1]]$.

Therefore, as a map into $\Jet^2(\frk g';V_3)$, we have
\[
\phi(v_1,v_2,v_3) =
( v_1,v_2;v_3;\frac12 [v_2,\cdot]; \frac12 [v_1,[y,x]] + \frac1{12} [x,[y,v_1]] + [y,[x,v_1]] ) ,
\]
where $x$ and $y$ are the place holders to indicate the ordered entries of the bilinear map.

\begin{remark}\label{rem617bb0a5}
	Notice that if $G=\R^n$ is abelian, then $\Jet^m(\R^n;W)$ has the property
	\begin{equation}\label{eq61ddacc5}
	[\jet^m(\R^n;W)_i,\jet^m(\R^n;W)_j]=0
	\qquad\text{ for all $i,j>1$.}
	\end{equation}
	Every stratified subgroup of $\Jet^m(\R^n;W)$ 
	must also satisfy~\eqref{eq61ddacc5}.
	Therefore, not all stratified Lie group of step larger than 2 can be embedded in a standard jet space.
\end{remark}

%%%%%%%%%%%%%%%%%%%%%%%%%%%%%%%%%%%%%%%%%%%%%%%%%%%%%%%%%%%%%%%
%%%%%%%%%%%%%%%%%%%%%%%%%%%%%%%%%%%%%%%%%%%%%%%%%%%%%%%%%%%%%%%
\section{Example: the first Heisenberg Group}
\label{sec617bad0c}

%%%%%%%%%%%%%%%%%%%%%%%%%%%%%%%%%%%%%%%%%%%%%%%%%%%%%%%%%%%%%%%
\subsection{The first Heisenberg group}
The Lie algebra $\frk h$ of the \emph{first Heisenberg group} $\bb H$ 
is the stratified three-dimensional Lie algebra with basis $X,Y,Z$ 
and with the only non-trivial bracket relation $Z=[X,Y]$.
The stratification $\frk h=V_1\oplus V_2$ is given by $V_1=\Span\{X,Y\}$ and $V_2=\R Z$.
In exponential coordinates $(x,y,z)$, we have
\[
\tilde X(x,y,z) = \de_x - \frac{y}{2} \de_z , \qquad
\tilde Y(x,y,z) = \de_y + \frac{x}{2} \de_z , \qquad
\tilde Z(x,y,z) = \de_z .
\]

We will apply the algorithm described in Remark~\ref{rem6093eb23} to obtain a basis for $\HorDer^m(\frk h;\R)$ for $m\in\{1,2,3\}$.
Moreover, we compute the corresponding basis of $\Poly_{e}^m(G;\R)$.
The corresponding basis for $\Poly_p^m(G;\R)$ can be obtained via~\eqref{eq61715bca}.

%%%%%%%%%%%%%%%%%%%%%%%%%%%%%%%%%%%%%%%%%%%%%%%%%%%%%%%%%%%%%%%
\subsection{A basis for $\HorDer^1(\frk h;\R)$}
We have already chosen the basis $\scr B=\{X,Y\}$ of $V_1$,
which has a dual basis $\{X^*,Y^*\}$ for $V_1^*$.
The set of multi-indices is $\scr I^1=\{(1,0,0),(0,1,0)\}$.
A basis of $\scr U^1(\bb H)$ is $\{\tilde X,\tilde Y\}$.
The map $\tau:\Tensor^1(V_1)\to\UniEnvAlg^1(\bb H)$ is given by $\tau(X)=\tilde X$ and $\tau(Y)=\tilde Y$, 
so by~\eqref{eq6178f2b6},
\[
A_{(1,0,0)} = X^* \quad {\rm and} \quad A_{(0,1,0)} = Y^*
\]
form a basis of $\HorDer^1(\frk h;\R)$.
The basis of $\Poly^1_{e}(\bb H;\R)$ dual to $\{\tilde X,\tilde Y\}$ is 
$\{x,y\}$.

%%%%%%%%%%%%%%%%%%%%%%%%%%%%%%%%%%%%%%%%%%%%%%%%%%%%%%%%%%%%%%%
\subsection{A basis for $\HorDer^2(\frk h;\R)$}
In this case, we have 
$\Xi^2 = \{X\otimes X, X\otimes Y, Y\otimes X, Y\otimes Y\}$.
The set of multi-indices is $\scr I^2 = \{(2,0,0),(1,1,0),(0,2,0),(0,0,1)\}$.
A basis of $\scr U^2(\bb H)$ is $\{\tilde X^2,\tilde X\tilde Y,\tilde Y^2,\tilde Z\}$.
We can compute the map $\tau:\Tensor^2(V_1)\to\UniEnvAlg^2(\bb H)$ as
\begin{align*}
\tau(X\otimes X) &= \tilde X\tilde X , &
\tau(X\otimes Y) &= \tilde Y\tilde X = \tilde X\tilde Y - \tilde Z , \\
\tau(Y\otimes X) &= \tilde X\tilde Y , &
\tau(Y\otimes Y) &= \tilde Y\tilde Y .
\end{align*}
Notice the specular order of $X$ and $Y$ on the two sides of the equality.
A basis of $\HorDer^2(\frk h;\R)$ is
\begin{align*}
A_{(2,0,0)} &= X^*\otimes X^* , &
A_{(1,1,0)} &= X^*\otimes Y^* + Y^*\otimes X^* , \\
A_{(0,2,0)} &= Y^*\otimes Y^* , &
A_{(0,0,1)} &= - X^*\otimes Y^* .
\end{align*}
Notice that $\HorDer^2(\frk h;\R) = \Tensor^2(V_1^*)$.

To compute a basis of $\Poly_{e}^2(\bb H;\R)$ dual to $\{\tilde X^2,\tilde X\tilde Y,\tilde Y^2,\tilde Z\}$ we need to first compute the action of each element of the latter basis to $\{x^2,xy,y^2,z\}$, which is a basis of $\Poly_{e}^2(\bb H;\R)$.
We present this action in the following table.

\newcolumntype{L}{>{$}c<{$}} % math-mode version of "c" column type
\renewcommand{\arraystretch}{1.2}
\begin{center}
\begin{tabular}{|L|L|L|L|L|}
\hline
 		& \tilde X^2 & \tilde X\tilde Y & \tilde Y^2 & \tilde Z \\ \hline
x^2 	& 2	& 0 & 0 & 0 \\ \hline
xy 	& 0 & 1 & 0 & 0 \\ \hline
y^2 	& 0 & 0 & 2 & 0 \\ \hline
z		& 0 & 1/2 & 0 & 1 \\ \hline
\end{tabular}
\end{center}
Therefore, the basis of $\Poly_{e}^2(\bb H;\R)$ dual to $\{\tilde X^2,\tilde X\tilde Y,\tilde Y^2,\tilde Z\}$ is 
$\{x^2/2,xy,y^2/2,z-xy/2\}$.

%%%%%%%%%%%%%%%%%%%%%%%%%%%%%%%%%%%%%%%%%%%%%%%%%%%%%%%%%%%%%%%
\subsection{A basis for $\HorDer^3(\frk h;\R)$}
We are now giving the formulas without description:
\begin{align*}
\Xi^3 &= \{
X\otimes X\otimes X, X\otimes X\otimes Y, X\otimes Y\otimes X, X\otimes Y\otimes Y, \\
&\qquad Y\otimes X\otimes X, Y\otimes X\otimes Y, Y\otimes Y\otimes X, Y\otimes Y\otimes Y
\} ; \\
\scr I^3 &= \{(3,0,0),(2,1,0),(1,2,0),(0,3,0),(1,0,1),(0,1,1)\} ; \\
	&\quad \{ \tilde X^3 , \tilde X^2\tilde Y, \tilde X\tilde Y^2, \tilde Y^3, \tilde X\tilde Z, \tilde Y\tilde Z\} ; \\
\tau(X\otimes X\otimes X) &= \tilde X\tilde X\tilde X 
	= \tilde X^3, \\
\tau(X\otimes X\otimes Y) &= \tilde Y\tilde X\tilde X 
	= \tilde X^2\tilde Y - 2\tilde X\tilde Z , \\
\tau(X\otimes Y\otimes X) &= \tilde X\tilde Y\tilde X 
	= \tilde X^2\tilde Y - \tilde X\tilde Z, \\
\tau(X\otimes Y\otimes Y) &= \tilde Y\tilde Y\tilde X 
	= \tilde X\tilde Y^2 - 2\tilde Y\tilde Z, \\
\tau(Y\otimes X\otimes X) &= \tilde X\tilde X\tilde Y
	= \tilde X^2\tilde Y, \\
\tau(Y\otimes X\otimes Y) &= \tilde Y\tilde X\tilde Y 
	= \tilde X\tilde Y^2 - \tilde Y\tilde Z , \\
\tau(Y\otimes Y\otimes X) &= \tilde X\tilde Y\tilde Y 
	= \tilde X\tilde Y^2, \\
\tau(Y\otimes Y\otimes Y) &= \tilde Y\tilde Y\tilde Y 
	= \tilde Y^3; \\
A_{(3,0,0)} &= X^*\otimes X^*\otimes X^* , \\
A_{(2,1,0)} &= X^*\otimes X^*\otimes Y^* + X^*\otimes Y^*\otimes X^* + Y^*\otimes X^*\otimes X^* , \\
A_{(1,2,0)} &= X^*\otimes Y^*\otimes Y^* + Y^*\otimes X^*\otimes Y^* + Y^*\otimes Y^*\otimes X^* , \\
A_{(0,3,0)} &= Y^*\otimes Y^*\otimes Y^* , \\
A_{(1,0,1)} &= -2 X^*\otimes X^*\otimes Y^* - X^*\otimes Y^*\otimes X^* , \\
A_{(0,1,1)} &= -2 X^*\otimes Y^*\otimes Y^* - Y^*\otimes X^*\otimes Y^* .
\end{align*}

\renewcommand{\arraystretch}{1.2}
\begin{center}
\begin{tabular}{|L|L|L|L|L|L|L|}
\hline
 		& \tilde X^3 & \tilde X^2\tilde Y & \tilde X\tilde Y^2 & \tilde Y^3 & \tilde X\tilde Z & \tilde Y\tilde Z \\ \hline
x^3 & 6 & 0 & 0 & 0 & 0 & 0 \\ \hline
x^2y & 0 & 2 & 0 & 0 & 0 & 0 \\ \hline
xy^2 & 0 & 0 & 2 & 0 & 0 & 0 \\ \hline
y^3 & 0 & 0 & 0 & 6 & 0 & 0 \\ \hline
xz & 0 & 1 & 0 & 0 & 1 & 0 \\ \hline
yz & 0 & 0 & 1 & 0 & 0 & 1 \\ \hline
\end{tabular}
\end{center}
Finally, from this table we get the basis of $\Poly_{e}^3(\bb H;\R)$ dual to\\ $\{ \tilde X^3 , \tilde X^2\tilde Y, \tilde X\tilde Y^2, \tilde Y^3, \tilde X\tilde Z, \tilde Y\tilde Z\}$:
\[
\left\{ 
\frac{x^3}{6} , 
\frac{x^2y}{2} ,
\frac{xy^2}{2} ,
\frac{y^3}{6} ,
xz - \frac{x^2y}{2} ,
yz - \frac{xy^2}{2}
\right\} .
\]

%%%%%%%%%%%%%%%%%%%%%%%%%%%%%%%%%%%%%%%%%%%%%%%%%%%%%%%%%%%%%%%
\subsection{The Lie algebra $\jet^2(\frk h;\R)$}\label{sec62d95272}
We will describe in Tables~\ref{tb62d94cbd} and~\ref{tb62d94fd6} the Lie algebra structure of 
\[
\jet^2(\frk h;\R) = \frk h \oplus \R \oplus \HorDer^1(\frk h;\R) \oplus \HorDer^2(\frk h;\R).
\]
The following computation for one of the Lie brackets might be instructive to the reader:
\begin{align*}
	[A_{(0,0,-1)},A_{(0,0,1)}]
	&= [Z,\dual{X}\ts\dual{Y}] \\
	&\overset{(*)}= - Z \rcontr (\dual{X}\ts\dual{Y}) \\
	&\overset{(**)}= [X\rcontr , Y\rcontr](\dual{X}\ts\dual{Y}) \\
	&= (X\rcontr Y\rcontr - Y\rcontr X\rcontr)(\dual{X}\ts\dual{Y}) \\
	&= (X\rcontr \dual{X}  - Y\rcontr 0) \\
	&= T
\end{align*}
Identity $(*)$ is an application of the definition~\eqref{eq5ebb173d},
while $(**)$ is an application of Proposition~\ref{prop05301400}.
The symbol $T$ denotes the standard basis element of $\R$, that is, $1$.

%%%%%%%%%%%%%%%%%%%%%%%%%%%%%%%%%%%%%%%%%%%%%%%%%%%%%%%%%%%%%%%
\newcolumntype{m}{>{$}c<{$}}

\begin{landscape}
\begin{table}
\caption{Reference table of explicit symbols used in Table~\ref{tb62d94cbd}}
\label{tb62d94fd6}

  \begin{tabular}{@{} |mmmmmm|mmm|m| @{}}
    \hline
     1 & 2 & 3 & 4 & 5 & 6 & 7 & 8 & 9 & 10 \\ 
    \hline
     A_{(-1, 0, 0)} & A_{(0, -1, 0)} & A_{(2, 0, 0)} & A_{(1, 1, 0)} & A_{(0, 2, 0)} & A_{(0, 0, 1)} & A_{(0, 0, -1)} & A_{(1, 0, 0)} & A_{(0, 1, 0)} & A_{(0, 0, 0)} \\ 
     X &
 Y &
 \dual{X}\ts\dual{X} &
 \dual{X}\ts\dual{Y} + \dual{Y}\ts\dual{X} &
 \dual{Y}\ts\dual{Y} &
 -\dual{X}\ts\dual{Y} &
 Z &
 \dual{X} &
 \dual{Y} &
 T \\
    \hline
  \end{tabular}
\end{table}
%%%%%%%%%%%%%%%%%%%%%%%%%%%%%%%%%%%%%%%%%%%%%%%%%%%%%%%%%%%%%%%%

\begin{table}

\vspace{1cm}

\caption{Lie bracket relations in $\jet^2(\frk h;\R)$.}
\label{tb62d94cbd}

\begin{tabular}{@{} |m|mmmmmm|mmm|m| @{}}
    \hline
    * & A_{(-1, 0, 0)} & A_{(0, -1, 0)} & A_{(2, 0, 0)} & A_{(1, 1, 0)} & A_{(0, 2, 0)} & A_{(0, 0, 1)} & A_{(0, 0, -1)} & A_{(1, 0, 0)} & A_{(0, 1, 0)} & A_{(0, 0, 0)} \\  
    \hline
    A_{(-1, 0, 0)} & 0 & A_{(0, 0, -1)} & -A_{(1, 0, 0)} & -A_{(0, 1, 0)} & 0 & 0 & 0 & -A_{(0, 0, 0)} & 0 & 0 \\
A_{(0, -1, 0)} & -A_{(0, 0, -1)} & 0 & 0 & -A_{(1, 0, 0)} & -A_{(0, 1, 0)} & A_{(1, 0, 0)} & 0 & 0 & -A_{(0, 0, 0)} & 0 \\
A_{(2, 0, 0)} & A_{(1, 0, 0)} & 0 & 0 & 0 & 0 & 0 & 0 & 0 & 0 & 0 \\
A_{(1, 1, 0)} & A_{(0, 1, 0)} & A_{(1, 0, 0)} & 0 & 0 & 0 & 0 & 0 & 0 & 0 & 0 \\
A_{(0, 2, 0)} & 0 & A_{(0, 1, 0)} & 0 & 0 & 0 & 0 & 0 & 0 & 0 & 0 \\
A_{(0, 0, 1)} & 0 & -A_{(1, 0, 0)} & 0 & 0 & 0 & 0 & -A_{(0, 0, 0)} & 0 & 0 & 0 \\
\hline
A_{(0, 0, -1)} & 0 & 0 & 0 & 0 & 0 & A_{(0, 0, 0)} & 0 & 0 & 0 & 0 \\
A_{(1, 0, 0)} & A_{(0, 0, 0)} & 0 & 0 & 0 & 0 & 0 & 0 & 0 & 0 & 0 \\
A_{(0, 1, 0)} & 0 & A_{(0, 0, 0)} & 0 & 0 & 0 & 0 & 0 & 0 & 0 & 0 \\
\hline
A_{(0, 0, 0)} & 0 & 0 & 0 & 0 & 0 & 0 & 0 & 0 & 0 & 0 \\ 
    \hline
  \end{tabular}
\end{table}
\end{landscape}
%%%%%%%%%%%%%%%%%%%%%%%%%%%%%%%%%%%%%%%%%%%%%%%%%%%%%%%%%%%%%%%

\end{document}